\documentclass[12pt,a4paper]{article}
\usepackage{amsmath,amsthm,amsfonts,amssymb,bbm}
\usepackage{graphicx,psfrag,subfigure,color}
\usepackage{cite}
\usepackage{hyperref}

\numberwithin{equation}{section}

\newcommand{\Or}{\mathcal{O}}

\newcommand{\Pb}{\mathbbm{P}}

\newcommand{\Id}{\mathbbm{1}}

\newcommand{\I}{{\rm i}}

\newcommand{\R}{\mathbb{R}}

\newcommand{\w}{\bf w}
\renewcommand{\b}{\bf b}
\newcommand{\Z}{\mathbb{Z}}

\renewcommand{\Im}{\mathrm{Im}}

\DeclareMathOperator{\sgn}{sign}

\newtheorem{prop}{Proposition}[section]
\newtheorem{thm}[prop]{Theorem}
\newtheorem{lem}[prop]{Lemma}
\newtheorem{defin}[prop]{Definition}
\newtheorem{cor}[prop]{Corollary}

\newtheorem{cla}[prop]{Claim}

\newtheorem{rem}[prop]{Remark}

\title{Speed and fluctuations for some  driven  dimer models }
\author{S. Chhita\thanks{Department of Mathematical Sciences, Durham University, Stockton Road, Durham, DH1 3LE, UK. E-mail: {\tt sunil.chhita@durham.ac.uk}} \and P.L. Ferrari\thanks{Institute for Applied Mathematics, Bonn University, Endenicher Allee 60, 53115 Bonn, Germany. E-mail: {\tt ferrari@uni-bonn.de}} \and F.L. Toninelli\thanks{CNRS and Institut Camille Jordan, Universit\'e Lyon 1, 43 bd du 11 novembre 1918, 69622 Villeurbanne, France. E-mail: {\tt toninelli@math.univ-lyon1.fr}}}
\date{}

\begin{document}

\maketitle
\sloppy
\begin{abstract}
  We consider driven dimer models on the square and honeycomb graphs,
  starting from a stationary Gibbs measure. Each model can be thought
  of as a two dimensional stochastic growth model of an interface,
  belonging to the anisotropic KPZ universality class. We use a
  combinatorial approach to determine the speed of growth and show
  logarithmic growth in time of the variance of the height function fluctuations.
\end{abstract}

\section{Introduction}\label{SectIntro}

We consider two-dimensional stochastic growth models in the
anisotropic KPZ universality class~\cite{Wol91}. Stochastic interface
growth models have a random local growth mechanism which is
(effectively) local in space and time, $t$, but with smoothing
mechanisms that ensure deterministic growth under hydrodynamic
scalings. In two dimensions, the average speed of growth $v(\rho)$ of the
interface in a stationary state can be parameterized by the slope
$\rho=(\rho_1,\rho_2)$ of the height function.  The anisotropic KPZ
universality class contains the models for which the signature of the
Hessian of the speed of growth is $(+,-)$. This is in contrast with
the usual, isotropic, KPZ universality class where the signature is
$(+,+)$ or $(-,-)$.  In the anisotropic case, it is expected that the
fluctuations of the height function behave asymptotically like
$\sqrt{\log t}$ as $t$ grows~\cite{Wol91}. This has been analytically verified for
some exactly solvable models~\cite{PS97,BF08,BF08b} and confirmed by numerical
studies~\cite{KKK98,HHA92}. Furthermore, it is expected that on large
space-time scales and modulo a linear transformation of space and time
coordinates, the height function fluctuations of the stationary
process have the same asymptotic correlations as those found in the
stochastic heat equation with additive noise (see~\cite{BCT15,BCF16}
for recent works).

In this paper we consider two dimer models on infinite bipartite
graphs $\Z^2$ and ${\cal H}$ (the honeycomb graph). Dimers, that are
viewed as particles, perform long-range jumps with asymmetric
rates. For the honeycomb graph, the dynamics were defined
in~\cite{BF08} and later extended to a partially asymmetric situation
in~\cite{Ton15}. The dynamics on $\mathbb Z^2$ was introduced in
\cite{Ton15}.  For both these models, translation-invariant stationary
measures for interface gradients are Gibbs measures on dimer
configurations with prescribed dimer
densities~\cite{BS08,KOS03,KenLectures}.  In~\cite{BF08}, the specific
prescribed initial conditions were not stationary but this choice had
the useful property that in a large enough subset of space-time, dimer
correlation functions were determinantal. This allowed, among others,
the computation of the law of large numbers and to determine that the
variance of the height function behaves asymptotically like $\log t$
and has Gaussian fluctuations on that scale.  However, this
determinantal property for space-time correlations, which allowed for
explicit computations, is no longer true for the partially asymmetric
dynamics or for those with stationary initial conditions.

In this paper we consider stationary initial conditions and obtain two
results, that apply equally to the totally asymmetric or to the
partially asymmetric situation. The first one is the speed of growth
$v^{\mathbb Z^2}(\rho)$ for the model on $\Z^2$
(Theorem~\ref{thm:square}). The difficulty here is to find a compact
and explicit formula for the speed of growth, since by definition of
the dynamics, $v^{\mathbb Z^2}(\rho)$  is given by an infinite sum of
probabilities of certain dimer configurations and therefore by an
infinite sum of determinants involving the inverse Kasteleyn
matrix. To obtain this result we mimic the approach
used for the honeycomb lattice in~\cite{CF15}. There, a combinatorial
argument showed that the infinite sum reduces to a single entry of the
inverse Kasteleyn matrix, leading to the explicit formula
(\ref{eq:cf}). For $\Z^2$, this is no longer the case, but we are able
to prove that the infinite sum is given in terms of a few explicit
entries of the inverse Kasteleyn matrix. As a side result, we verify
explicitly that the signature of the Hessian of
$v^{\mathbb Z^2}(\rho)$ is $(+,-)$.

The second result concerns the logarithmic growth of variance of the
height function for the honeycomb graph, see Theorem~\ref{th:varianza}
(the method can be extended to the dynamics on $\mathbb Z^2$ but in
order not to overload this work we skip this). This result was
partially proved in~\cite{Ton15}, with a technical restriction on
slope $\rho$. Our new approach simplifies the proof contained
in~\cite{Ton15} and it extends its domain of validity to the full set
of allowed slopes.

The rest of the paper is organized as follows. In Section~\ref{sectResults} we define the models and give the results. Section~\ref{sectBackground} contains the background on dimers models. Theorem~\ref{thm:square} on the speed of growth on $\Z^2$ is proved in Section~\ref{sec:vz2}. Theorem~\ref{th:varianza} on the variance is proved in Section~\ref{sectThmVar}.

\subsubsection*{Acknowledgements}
F.  T.  was partially funded by the ANR-15-CE40-0020-03 Grant LSD, by the CNRS
PICS grant ``Interfaces al\'eatoires discr\`etes et dynamiques de
Glauber'' and by MIT-France Seed Fund ``Two-dimensional Interface
Growth and Anisotropic KPZ Equation''. P.F. was supported by the German Research Foundation as part of the SFB 1060--B04 project.

\section{The growth models and the results}\label{sectResults}

\subsection{Perfect matchings and height function}
We are interested in two infinite, bipartite planar graphs
$\mathcal G=(\mathcal V,\mathcal E)$ in this work: the grid $\mathbb Z^2$ and the honeycomb
lattice $\mathcal H$. In both cases, we let $\mathcal M_{\mathcal G}$
denote the set of perfect matchings or dimer coverings of
$\mathcal G$, i.e., subsets of edges in $\mathcal E$ (dimers) such that each vertex is
incident to exactly one edge. Both graphs are bipartite, so we can fix
a $2$-coloring (say, black and white) of their vertices $\mathcal V$, see
Figures~\ref{fig:hex} and~\ref{fig:threads}. We denote $W_{\mathcal G}$ (resp.\ $B_{\mathcal G}$) to be the set of white (resp.\ black) vertices of $\mathcal G$.

Associated to each dimer covering $m\in \mathcal M_{\mathcal G}$, there  is a
height function $h$ defined on faces of $\mathcal G$, as follows: $h$
is fixed to zero at some given face $x_0$ of $\mathcal G$ (the
``origin'') and its gradient are given by
 \begin{equation}
\label{deltaH}
 h(x)-h(y)=\sum_{e\in C_{x\to y}} \sigma_e({\bf 1}_{e\in m}-c(e))
 \end{equation}
where: $x,y$ are faces of $\mathcal G$, $C_{x\to y}$ is any nearest-neighbor path from $x$ to $y$ (the r.h.s.\ of \eqref{deltaH} does not depend on the choice of $C_{x\to y}$), the sum runs over edges crossed by $C_{x\to y}$,
$\sigma_e$ equals $+1$ (resp.\ $-1$) if $e$ is crossed with the white vertex on the right (resp.\ left) and $c(\cdot)$ is a function defined on the edges of $\mathcal G$, such that for any $v\in\mathcal V$,
 \begin{equation}
\sum_{e:e\sim v}c(e)=1,
 \end{equation}
where $e\sim v$ means that $e$ is incident to $v$.  A standard choice
for the square lattice is $c(e)\equiv 1/4$; for the hexagonal lattice,
we let $c(e)=1$ if $e$ is horizontal and $c(e)=0$ otherwise.

As we recall in more detail in Section~\ref{subsec:Gibbs} below, for
both graphs there exists an open polygon
$P^{\mathcal G}\subset \mathbb R^2$ such that for every
$\rho=(\rho_1,\rho_2)\in P^{\mathcal G}$ there exists a unique
translation invariant and ergodic Gibbs probability measure on dimer
coverings of $\mathcal G$, denoted $\pi^{\mathcal G}_\rho$. With the choice of coordinates we make in this work (see Section \ref{subsec:Gibbs}), the
polygons $P^{\mathcal G}$ for the two graphs are as follows:
\begin{defin}
\label{def:NP}
$P^{\mathcal H}$ is the open triangle in $\mathbb R^2$ with vertices
$(0,0),(0,1),(1,1)$, and $P^{\mathbb Z^2}$ is the open square in
$\mathbb R^2$ with vertices $(\pm1/2,\pm1/2)$.
\end{defin}

\subsection{Particles and interlacement conditions}

A common feature of the two graphs $\mathbb Z^2$ and $\mathcal H$,
that makes them special with respect to other planar, bipartite
graphs, is that to any $m\in\mathcal M_{\mathcal G}$ one can associate
a collection of ``interlaced particles''. First of all, we partition the set
of faces of $\mathcal G$ into disjoint ``columns''
$C_\ell,\ell\in\mathbb Z$. In the case of $\mathcal H$, a column
$C_\ell$ consists in the set of faces with the same horizontal
coordinate, while for $\mathbb Z^2$ it is a zig-zag path as depicted
in Fig~\ref{fig:threads}.
\begin{figure}
\begin{center}
\includegraphics[height=6cm]{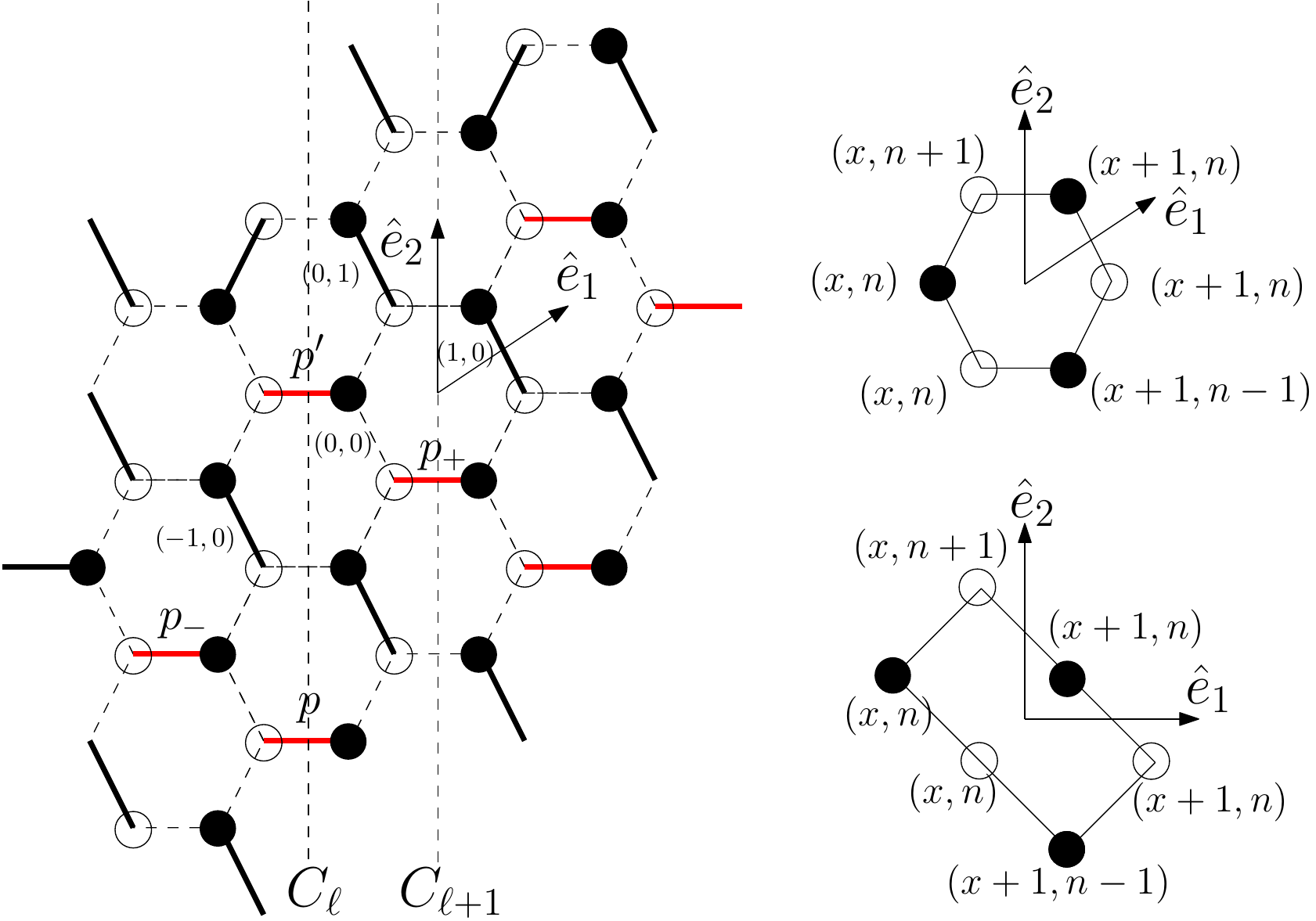}
\caption{The left figure shows the hexagonal graph $\mathcal H$ with the axes
  $\hat e_1,\hat e_2$ and the columns $C_\ell$. Coordinates
  $(x_1,x_2)$ are the same for the black and white vertices on the
  same north-west oriented edge. Particles (i.e., horizontal dimers)
  are marked in red. Particles $p,p'$ on column $C_{\ell}$ are
  vertically interlaced with particles $p_\pm$ on $C_{\ell\pm1}$. In
  Section~\ref{sec:varianza}, we will re-draw hexagonal faces as
  rectangular ones, as in the drawing on the  right side.}
\label{fig:hex}
\end{center}
\end{figure}
\begin{figure}
\begin{center}
\includegraphics[height=7cm]{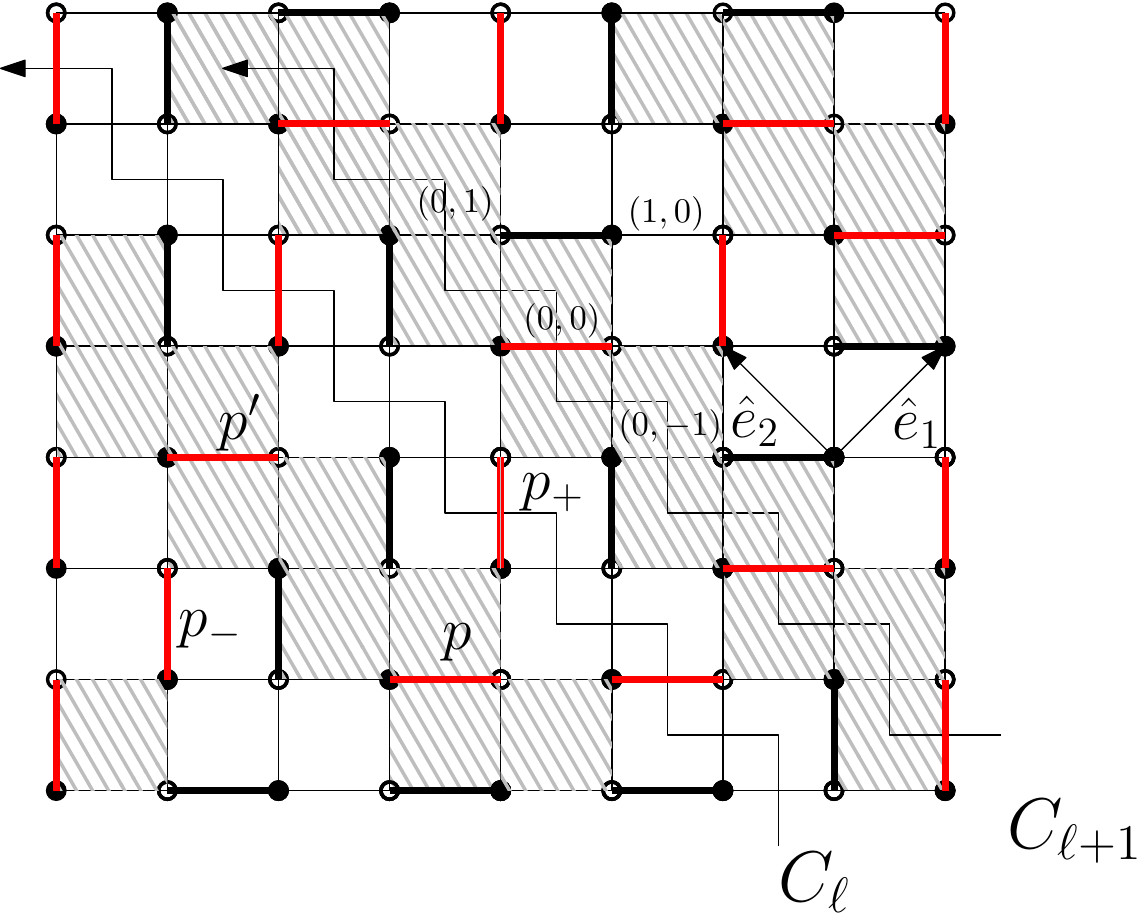}
\caption{The square lattice $\mathbb Z^2$ with the axes $\hat e_1,\hat e_2$ and the ``columns'' $C_\ell$. Coordinates
$(x_1,x_2)$ are the same for a black and the white vertex just to its right. Thick edges are dimers, and transversal dimers (or particles) are drawn in red.}
\label{fig:threads}
\end{center}
\end{figure}
We call $Y_\ell$ the set of vertices of $\mathcal G$ shared by $C_\ell$ and
 $C_{\ell+1}$.
 Vertices $v\in Y_\ell$ can be ordered in a natural way and we will say that $v_1<v_2$ if $v_1$ precedes
$v_2$ in the upward direction (for $\mathcal H$) or in the up-left direction of Figure~\ref{fig:threads} (for $\mathbb Z^2$).
An edge $e$ of $\mathcal G$ will be called ``transversal'' if it has one endpoint on $Y_{\ell}$ and the other
on $Y_{\ell+1}$ for some $\ell$. Dimers on transversal edges will be called ``particles''.

Given two particles $p$ and $p'$, each with one endpoint (say $v,v'$
respectively) on the same $Y_\ell$, let us say that $p'$ is higher
than $p$ (we write $p<p'$) if $v<v'$. The following interlacement
condition is easily verified both for $\mathcal H$ and $\mathbb Z^2$:
given two particles $p,p'$ on the same column $C_\ell$ and verifying
$p<p'$, there exists a particle $p_-$ on $C_{\ell-1}$ and a particle
$p_+$ on $C_{\ell+1}$ such that $p<p_-<p'$, $p<p_+<p'$. See Figures
\ref{fig:hex} and~\ref{fig:threads}.

Both on $\mathbb Z^2$ and on
$\mathcal H$ it is easy to check that, under the assumption that every
$C_\ell$ contains at least one particle, the whole dimer configuration
is uniquely determined by the particle configurations. In the
situation we are interested in, there are almost surely infinitely
many particles on each $C_\ell$; therefore, we will implicitly
identify a dimer configuration $m\in\mathcal M_{\mathcal G}$ and the
corresponding particle configuration.

\subsection{Dynamics and new results}
We describe here the growth dynamics of \cite{Ton15} in a unified way for
$\mathcal G=\mathbb Z^2$ and $\mathcal G=\mathcal H$. We need some preliminary notation. Given a
transversal edge $e$ on column $C_\ell$, let $p(e)$ denote the highest
particle in column $C_\ell$ that is strictly below $e$. Given a
configuration $m\in\mathcal M_{\mathcal G}$, we say that ``particle $p(e)$ can reach edge $e$'' if the
configuration $m'$ obtained by moving $p(e)$ to edge $e$ while all other
particles positions are unchanged still satisfies the particle interlacement constraints, i.e., $m'\in \mathcal M_{\mathcal G}$.

The continuous time Markov chain of \cite{Ton15}, in its totally asymmetric
version, can be informally described as follows.  To each transversal
edge $e$ of $\mathcal G$ is associated an i.i.d. exponential clock of
mean $1$. When the clock at $e$ rings, if particle $p(e)$ can reach
$e$ without violating the interlacement constraints then it is moved
there. If $p(e)$ cannot reach $e$, then nothing happens.

Note that the size of particle jumps are unbounded, so it is not
a-priori obvious that the definition of the Markov process is
well-posed. However, one of the results of \cite{Ton15} is that given any
$\rho\in P^{\mathcal G}$, for almost every initial condition sampled
from the Gibbs measure $\pi^{\mathcal G}_\rho$ the dynamics is
well-defined (i.e., almost surely no particle travels an infinite
distance in finite time). Also, it is proved there that the measures
$\pi^{\mathcal G}_\rho$ are stationary for the dynamics. We let
$\nu^{\mathcal G}_\rho$ denote the law of the stationary process
started from $\pi^{\mathcal G}_\rho$.

Note that when $p(e)$ is moved from its current position in column say $C_\ell$ to the edge $e$ in the same column, it jumps over a certain number $n\ge 1$ of faces of $C_\ell$.
We define the ``integrated current'' $J(t)$ as the total number of particles that jump across a given face of the graph, say across the face $x_0$ that was chosen as origin, from time $0$ to time $t$ ($J(t)$  is trivially related to  the height change at $x_0$).
In \cite{Ton15} it was proven:
\begin{thm}
 For every $\rho \in P^{\mathcal G}$, there exists $v^{\mathcal G}(\rho)>0$ such that
 \begin{equation}
 \label{eq:veloc}
\nu^{\mathcal G}_\rho(J(t))=t v^{\mathcal G}(\rho).
 \end{equation}
Moreover, if $\mathcal H=\mathcal G$ then there exists a non-empty subset of $A\subset P^{\mathcal G}$ such that,
for every $\rho\in A$,
\begin{equation}
 \label{eq:var1}
\limsup_{t\to\infty} \frac{{\rm Var}_{{\nu^{\mathcal G}_\rho}}(J(t))}{\log t}<\infty.
\end{equation}
\end{thm}
Later, in \cite{CF15}, the function $v^{\mathcal G}(\rho)$ for $\mathcal G=\mathcal H$ was computed explicitly\footnote{In this work we use different conventions as in \cite{Ton15} for lattice coordinates and this is the reason why formula
\eqref{eq:cf} looks different from formula (3.6) of \cite{Ton15}}:
\begin{equation}
 \label{eq:cf}
 v^{\mathcal H}(\rho)=\frac1\pi\frac{\sin(\pi\rho_1)\sin(\pi (\rho_2-\rho_1))}{\sin(\pi\rho_2)}.
\end{equation}

Our main results here complete the above picture as follows:
\begin{thm} \label{thm:square}
For the dynamics on $\mathbb{Z}^2$, the speed of growth is given by
\begin{equation}
 \label{vz2}
 v^{\mathbb Z^2}(\rho) =\frac{1}{\pi}\sin\psi_1\left(\frac{ \sin\psi_1}{ \tan\psi_2} +\sqrt{ 1+\frac{ \sin^2\psi_1}{ \tan^2\psi_2}}\right)
 \end{equation}
	where $\psi_i=(\rho_i+1/2)\pi$ for $i \in \{1,2\}$ takes value in $[0,\pi]$.
\end{thm}
It is immediate to see that the r.h.s.\ of \eqref{eq:cf} (resp.\ of
\eqref{vz2}) is positive in the whole triangle $P^{\mathcal H}$
(resp.\ in the square $P^{\mathbb Z^2}$) of Definition~\ref{def:NP}.
\begin{thm}
\label{th:varianza}
For $\mathcal G=\mathcal H$ , \eqref{eq:var1} holds for every $\rho\in P^{\mathcal G}$.
\end{thm}
Moreover, the proof of \eqref{eq:var1} we give here is substantially simplified w.r.t. the one in \cite{Ton15}.
Also, our method can be easily adapted to prove Theorem~\ref{th:varianza} also for the dynamics on $\mathbb Z^2$ and every $\rho\in P^{\mathbb Z^2}$ but, in order to keep this work within a reasonable length, we do not give details on this extension.

\begin{rem}
  From the above explicit expression \eqref{vz2} it is possible to
  check (see Appendix~\ref{app:Hessian}) that the Hessian of the
  function $\rho\mapsto v^{\mathbb Z^2}(\rho)$ has signature $(+,-)$
  for every $\rho\in P^{\mathbb Z^2}$. This means that our model belongs
  to the anisotropic KPZ universality class.
\end{rem}

\begin{rem}
  The work~\cite{Ton15} studies a
  more general, partially asymmetric dynamics where upward jumps have rate $p$ and downward jumps have rate $q$. In this case, the speed of growth is given by the above formulas multiplied by $p-q$. Also, the result on the variance holds true also for the partially asymmetric version. In fact, from \cite[Sec. 9]{Ton15} one sees that Theorem \ref{th:varianza} holds for general $p,q$ as soon as
  Theorem \ref{th:varV} below, that is independent of $p,q$, is proved.
\end{rem}

\subsection{Geometric interpretation of $v^{\mathcal G}(\rho)$}
The stationary and translation invariant Gibbs measures form a
two-parameter family. From an interface perspective, it is natural to
use the average slope of the interface, $\rho=(\rho_1,\rho_2)$, as
parametrization. Then all other quantities, such as the average number
of dimers of a given type or the speed of growth, are functions of
$\rho$.  As it was already known for the honeycomb lattice, the
correlation kernel giving dimer correlations, that in principle is a double contour integral \cite{KOS03}, can be rewritten as a
single integral from $\overline \Omega_c$ to $\Omega_c$, where
$\Omega_c=\Omega_c(\rho)$ is a complex number in the upper half plane
$\mathbb{H}$. Further, for ${\cal G}={\cal H}$ and for a special
initial condition, it was shown~\cite{BF08} that the height field
fluctuations of the growth model converges to a Gaussian free field
(GFF).  More precisely, the correlations on a macroscopic scale at $m$
different points converge to the correlations of the GFF on
$\mathbb{H}$ between the points obtained by mapping the $m$ points to
$\mathbb{H}$ by $\Omega_c$. The map $\Omega_c$ was known already from
the work of Kenyon~\cite{Ken04} (there it is called $\Phi$ in Section~1.2.3
and Figure~2).  A generalization of~\cite{BF08} to a setting with two
different jump rates was made in~\cite{Dui11}.

Here we shortly present how the densities of the different types of
dimers, the correlation kernel and the speed of growth are written in
terms of $\Omega_c$. For the hexagonal lattice we refer
to~\cite{BF08}: the three types of dimers are in Figure~5.1, the
points $0$, $1$ and $\Omega_c$ form a triangle whose internal angles
are $\pi$ times the frequencies of the types of dimers (Figure~3.1),
and the correlation kernel as a single integral is given in~\cite[Prop. 3.2]{BF08}. Finally, an interesting property is that the speed of
growth (\ref{eq:cf}) equals $\frac{1}{\pi}\Im(\Omega_c)$.

For the square lattice, there also exists
$\Omega_c=\Omega_c(\rho)\in\mathbb{H}$ (not the same one as for the
hexagonal lattice) such that the correlation kernel is given as a
single integral from $\overline \Omega_c$ to $\Omega_c$ (see
Lemma~\ref{appendixlemma}). Using this and formula
\eqref{eq:localstatsG} below, one can easily compute the densities of
the different types of dominoes with the result
\begin{equation}\label{eqGeom1}
\begin{aligned}
a_1&=\textrm{density of }(\bullet(0,0),\circ(0,0))=\frac{1}{\pi} \left[\arg(\Omega_c)-\arg(\Omega_c+1)\right],\\
a_2&=\textrm{density of }(\bullet(0,0),\circ(0,1))=\frac{1}{\pi} \left[\arg(\Omega_c-1)-\arg(\Omega_c)\right],\\
a_3&=\textrm{density of }(\bullet(0,0),\circ(-1,1))=1-\frac{1}{\pi} \arg(\Omega_c-1),\\
a_4&=\textrm{density of }(\bullet(0,0),\circ(-1,0))=\frac{1}{\pi} \arg(\Omega_c+1).\\
\end{aligned}
\end{equation}
Further, the slopes are given (see (\ref{eq:Zslope1})-(\ref{eq:Zslope2}) and Lemma~\ref{lem:Z2slopes}) by
\begin{equation}\label{eqGeom2}
\begin{aligned}
\rho_1+\tfrac12 &= \frac{1}{\pi}\arg(\Omega_c) = a_1+a_4=1-a_2-a_3,\\
\rho_2+\tfrac12 &= \frac{1}{\pi}\left[\arg(\Omega_c-1)-\arg(\Omega_c+1)\right]=a_1+a_2=1-a_3-a_4.
\end{aligned}
\end{equation}
Finally, it follows from Appendix \ref{App:Positivity} that the speed of growth \eqref{vz2} is given by
\begin{equation}
v^{\Z^2}(\rho)=\frac{1}{\pi}\Im(\Omega_c),
\end{equation}
which, remarkably, is the same form as in the hexagonal case. As in the hexagonal case, also in the square case $\Omega_c(\rho)$ has  a nice geometric representation in terms of dimer densities, see Figure~\ref{fig:Geometry}.
\begin{figure}
\begin{center}
\includegraphics[height=5cm]{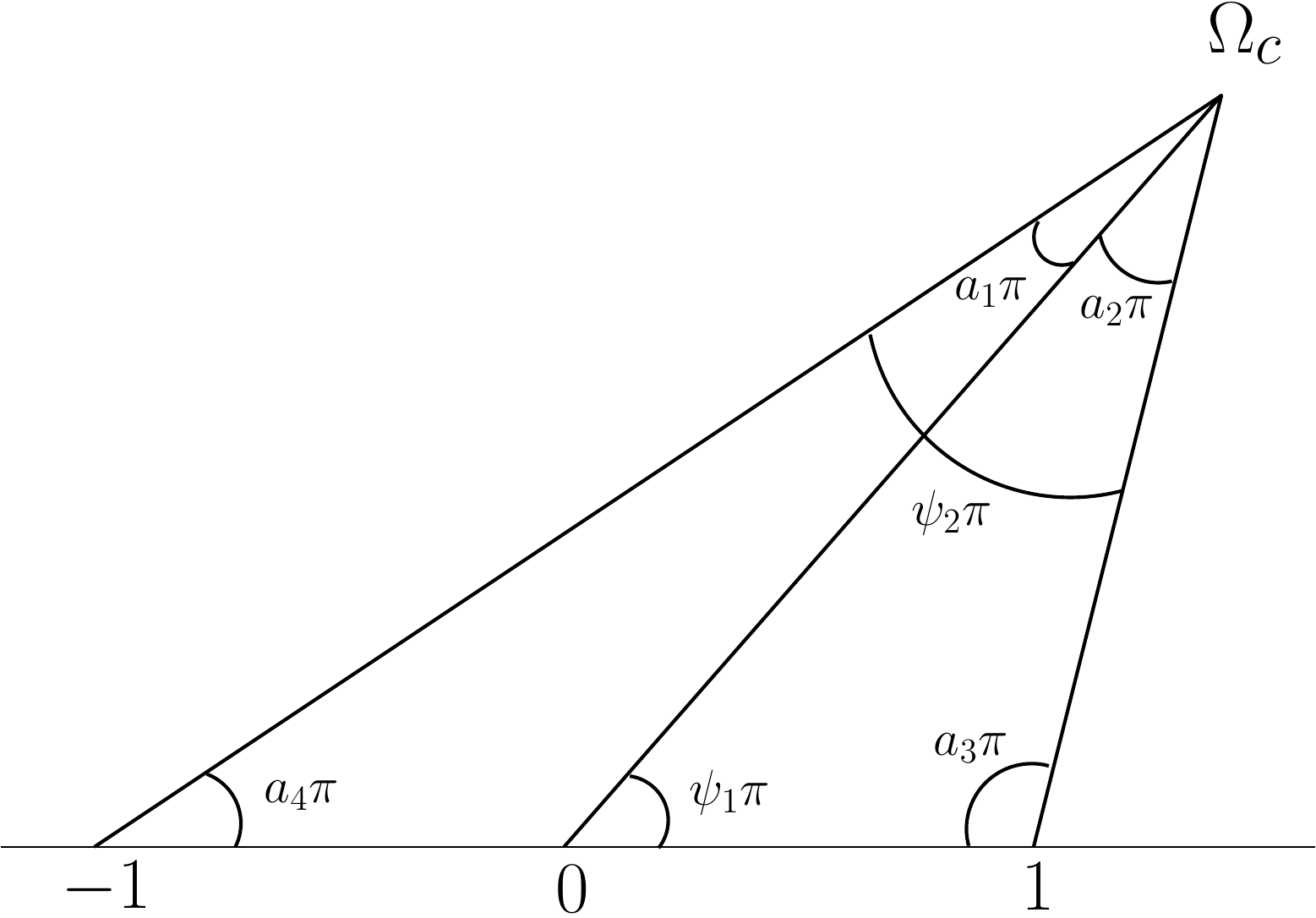}
\caption{Geometric interpretation of  $\Omega_c$ in terms of slopes and dimers densities. Here $\psi_i=\rho_i+1/2$.}
\label{fig:Geometry}
\end{center}
\end{figure}

\section{Background}\label{sectBackground}

\subsection{Gibbs Measures} \label{subsec:Gibbs}

An ergodic Gibbs measure or simply a Gibbs measure $\pi$, in our
context, is a probability measure on $\mathcal M_{\mathcal G}$ that is
invariant and ergodic w.r.t. translations in $\mathcal G$ and
satisfies the following form of DLR (Dobrushin-Lanford-Ruelle)
equations: for any finite subset of edges $\Lambda$, the law
$\pi(\cdot|m_{\Lambda^c})$ conditioned on the dimer configuration on edges
not in $\Lambda$ is the uniform measure on the finitely many dimer
configurations on $\Lambda$ that are compatible with $m_{\Lambda^c}$.
By translation invariance, to a Gibbs measure one can associate an average slope $\rho=(\rho_1,\rho_2)$, such that
\begin{equation}
\pi(h(x+\hat e_i)-h(x))=\rho_i,\quad i=1,2,
\end{equation}
with $\hat e_i$ the coordinate unit vectors.

It is convenient, both for this section and the rest of the work, to
make an explicit choice of coordinates on $\mathcal G$. Let us start
with the graph $\mathcal H$. Both white and black vertices are
assigned coordinates $x=(x_1,x_2)\in\mathbb Z^2$. The two (white and
black) endpoints of the same north-west oriented edge will be assigned
the same coordinates (we will denote them
$\circ(x_1,x_2),\bullet(x_1,x_2)$) and we make an arbitrary choice of
which edge has endpoints of coordinates $(0,0)$. The coordinate vectors
$\hat e_1,\hat e_2$ are chosen to be the unit vectors forming an angle
$\pi/6$ and $\pi/2$, respectively, w.r.t. the horizontal axis. See
Figure~\ref{fig:hex}. Note that the nearest neighbors of the black
vertex $\bullet(0,0)$ are the white vertices $\circ(0,0),\circ(0,1)$ and $\circ(-1,1)$.

As for $\mathbb Z^2$, we let $\hat e_1, \hat e_2$ be the vectors
forming an angle $\pi/4$ and $3\pi/4$ w.r.t. the horizontal axis,
see Figure~\ref{fig:threads}. Again we fix arbitrarily the origin of the lattice and
we establish that a white vertex has the same coordinates $(x_1,x_2)$
as the black vertex just to its left. The nearest neighbors of the black vertex $\bullet(0,0)$ are the white vertices
$\circ(0,0),\circ(0,1),\circ(-1,1),\circ(-1,0)$.

Recalling the definition of height function it is easy to see that,
for any Gibbs measure $\pi$, the slope $\rho$ must belong to the
closure of the polygon $ P^{\mathcal G}$ of Definition~\ref{def:NP}.

It is known \cite{KOS03} that for every $\rho \in P^{\mathcal G}$ there
exists a unique Gibbs measure $\pi:=\pi^{\mathcal G}_\rho$ with slope
$\rho$. This can be obtained as the limit (as $L\to\infty$) of the
uniform measure on the subset of dimer coverings of the $L\times L$
periodization of the lattice $\mathcal G$ such that the height
function changes by $\lfloor L \rho_i\rfloor$ along a cycle in
direction $\hat e_i,i=1,2$.

The correlations of the measure $\pi^{\mathcal G}_\rho$ have a
determinantal representation \cite{KOS03}, that we briefly recall here. First of all, one needs to
introduce the Kasteleyn matrix: this is the infinite, translation
invariant, matrix
$\left[\bar K(b,w)\right]_{b\in B_{\mathcal G},w\in W_{\mathcal G}}$, with
rows/columns indexed by black/white vertices of $\mathcal G$. Matrix
elements are non-zero complex numbers for $b,w$ nearest neighbors and
are zero otherwise. The non-zero elements depend also on the slope
$\rho$. See below for the explicit expression of $\bar K$ for the graphs
$\mathcal H$ and $\mathbb Z^2$. Next, one introduces an infinite,
translation-invariant matrix
$\left[\bar K^{-1}(w,b)\right]_{w\in W_{\mathcal G},b\in B_{\mathcal G}}$ (as
the notation suggests,  $\bar K \bar K^{-1}$
equals the identity matrix). Again, see below for the expression of
$\bar K^{-1}$ for $\mathcal G=\mathcal H$ and $\mathcal G=\mathbb Z^2$. All multi-point
correlations of $\pi^{\mathcal G}_\rho$ can be expressed via $\bar K $
and $\bar K^{-1}$ as follows \cite{KOS03}: given edges $e_i=(w_i,b_i), i\le k$,
\begin{equation}\label{eq:localstatsG}
\pi^{\mathcal G}_\rho(e_1,\dots,e_k\in m) = \bigg(\prod_{i=1}^k \bar K(b_i,w_i)\bigg) \det[\bar K^{-1}(w_i,b_j)]_{1 \leq i,j \leq k}.
\end{equation}

The definition of matrices $\bar K,\bar K^{-1}$ is not unique and different choices than the one we make below can be found in the literature.

For $\mathcal G=\mathcal H$, we let
\begin{equation}
 \label{eq:KH}
	\overline{K}(b,w)=\left\{
 \begin{array}{ll}
 a_2 & \text{ if } b=\bullet(x_1,x_2),w=\circ(x_1,x_2),\\
a_1 &\text{ if } b=\bullet(x_1,x_2),w=\circ(x_1-1,x_2+1),\\
a_3 &\text{ if }b=\bullet(x_1,x_2),w=\circ(x_1,x_2+1),
 \end{array}
 \right.
\end{equation}
where $a_i=a_i(\rho)>0$ are such that in the triangle with sides $a_1,a_2,a_3$, the angle opposite to the side of length $a_i$ is $\pi r_i>0$,
with $r_1=1-\rho_2$, $r_2=\rho_1$ and $r_3=\rho_2-\rho_1$. Note that $r_1$ (resp.\ $r_2,r_3$) is the density of dimers oriented horizontally (resp.\ oriented north-west, north-east). The inverse Kasteleyn matrix $\overline{K}^{-1}$ is
\begin{equation} \label{hexagonK-1}
\overline{K}^{-1}(w,b)=\frac{1}{(2\pi \mathrm{i})^2} \int dz_1 dz_2 \frac{ z_1^{y_2-x_2} z_2^{x_2-y_2+x_1-y_1-1}}{a_1+a_2z_1+a_3z_2},
\end{equation}
with $w=\circ(x_1,x_2)$ and $b=\bullet(y_1,y_2)$ and the integral runs over the anticlockwise circles $|z_1|=|z_2|=1$ in the complex plane.

For $\mathcal G=\mathbb Z^2$ we take instead
\begin{equation}
 \label{eq:KZ}
	\overline{K}(b,w)=\left\{
 \begin{array}{ll}
 \mathrm{i} e^{B_1} & \text{ if }b=\bullet(x_1,x_2),w=\circ(x_1,x_2),\\
 e^{B_1+B_2} &\text{ if } b=\bullet(x_1,x_2),w=\circ(x_1,x_2+1),\\
 \mathrm{i} e^{B_2} &\text{ if }b=\bullet(x_1,x_2),w=\circ(x_1-1,x_2+1),\\
 1 &\text{ if }b=\bullet(x_1,x_2),w=\circ(x_1-1,x_2),
 \end{array}
 \right.
\end{equation} (the ``magnetic fields'' $B_1,B_2$ are fixed by the slope $\rho$ as specified below)
and the inverse Kasteleyn matrix $\overline K^{-1}$ is given by
\begin{equation} \label{eq:wholeplaneK}
	\overline{K}^{-1}(w,b)=\frac1{(2\pi \mathrm{i})^2}\int \frac{d z_1}{z_1}\frac{d z_2}{z_2}
\frac{z_1^{y_1-x_1}z_2^{y_2-x_2}}{\mu(z_1,z_2)},\quad w=\circ(x_1,x_2),b=\bullet(y_1,y_2)
\end{equation}
where the integral runs over $|z_1|=|z_2|=1$ and
\begin{equation}
 \label{mu2}
\mu(z_1,z_2)= z_1( 1+ e^{B_1} \mathrm{i}z_1^{-1} +e^{B_2}\mathrm{i}z_2^{-1}+e^{B_1+B_2}z_1^{-1} z_2^{-1}),
\end{equation}
The parameters $B=(B_1,B_2)$ are related to the slope $\rho=(\rho_1,\rho_2)$ as follows:
\begin{equation} \label{eq:Zslope1}
\rho_1=\rho_1(B)= \frac{1}{2} -\left(\mathrm{i}e^{B_2} \overline{K}^{-1}(\circ(-1,1),\bullet (0,0))+e^{B_1+B_2}\overline{K}^{-1}(\circ(0,1),\bullet (0,0))\right)
\end{equation}
and
\begin{equation}
\label{eq:Zslope2}
\rho_2=\rho_2(B)= -\frac{1}{2} +\left(\mathrm{i}e^{B_1} \overline{K}^{-1}(\circ(0,0),\bullet (0,0))+e^{B_1+B_2}\overline{K}^{-1}(\circ(0,1),\bullet (0,0))\right).
\end{equation}
This is simply because, by the definition of height function, one has for instance
\begin{equation}
\rho_1=1/2-\pi^{\mathbb Z^2}_\rho((\circ(-1,1),\bullet (0,0))\in m)-
\pi^{\mathbb Z^2}_\rho((\circ(0,1),\bullet (0,0))\in m)
\end{equation}
and then \eqref{eq:Zslope1} follows from \eqref{eq:localstatsG}.
It is known \cite{KOS03} that the relations \eqref{eq:Zslope1},
\eqref{eq:Zslope2} given a bijection between $ P^{\mathbb Z^2}$ and the ``amoeba''
\begin{equation}
 \label{eq:ameba}
 \mathcal B= \{B:|\sinh(B_1)\sinh(B_2)|< 1\}=\{B:|\tanh(B_2)\cosh(B_1)|\le 1\}.
\end{equation}
Injectivity of the map $B\mapsto \rho(B)$ is related to the  fact that  $(\rho_1,\rho_2)$ is the gradient w.r.t. $(B_1,B_2)$ of a surface tension function that is a convex function of $(B_1,B_2)$.

\subsection{Average and variance of the current}
For ease of notation, given distinct edges $e_1,\dots e_{n+k}$ of $\mathcal G$, we let
\begin{equation}
\pi^{\mathcal G}_\rho(e_1,\dots,e_n,e_{n+1}^c,\dots,e_{n+k}^c):=\pi^{\mathcal G}_\rho(e_i\in m\;\forall i\le n,\; e_{n+i}\not\in m \;\forall i\le k)
\end{equation}
where we recall that $m$ denotes the dimer covering.

\subsubsection{Average current}
\label{sec:avc}
Let $\ell_0$ denote the index such that the face $x_0$ of $\mathcal G$
that we established to be the origin is in column $C_{\ell_0}$ and
let $S$ denote the set of edges $e$, transversal to $C_{\ell_0}$, that are
above $x_0$. Then, from the definition of the dynamics, we obtain the
following expression for the speed $v^{\mathcal G}(\rho)$:
\begin{equation}
 \label{vz1}
v^{\mathcal G}(\rho)=\sum_{e\in S} \pi^{\mathcal G}_\rho(U(e)):=\sum_{e\in S} \pi^{\mathcal G}_\rho (p(e) \text { is below $x_0$ and  can reach edge } e)
\end{equation}
simply because, if $p(e)$ can reach $e$, it will do so with rate $1$
and with such an update, it will increase the integrated current $J(t)$
through $x_0$ by $1$.

Then, \eqref{vz1} can be expressed more explicitly. In view of Theorem~\ref{thm:square}, we consider the case $\mathcal
G=\mathbb Z^2$. With reference to Figure~\ref{fig:coords}, where for
convenience we rotated the graph by $\pi/4$ clockwise, we first notice
that $S$ consists of the set of transversal edges $\{\bar e_i\}_{i\ge 1}$.
\begin{figure}
\begin{center}
\includegraphics[height=8cm]{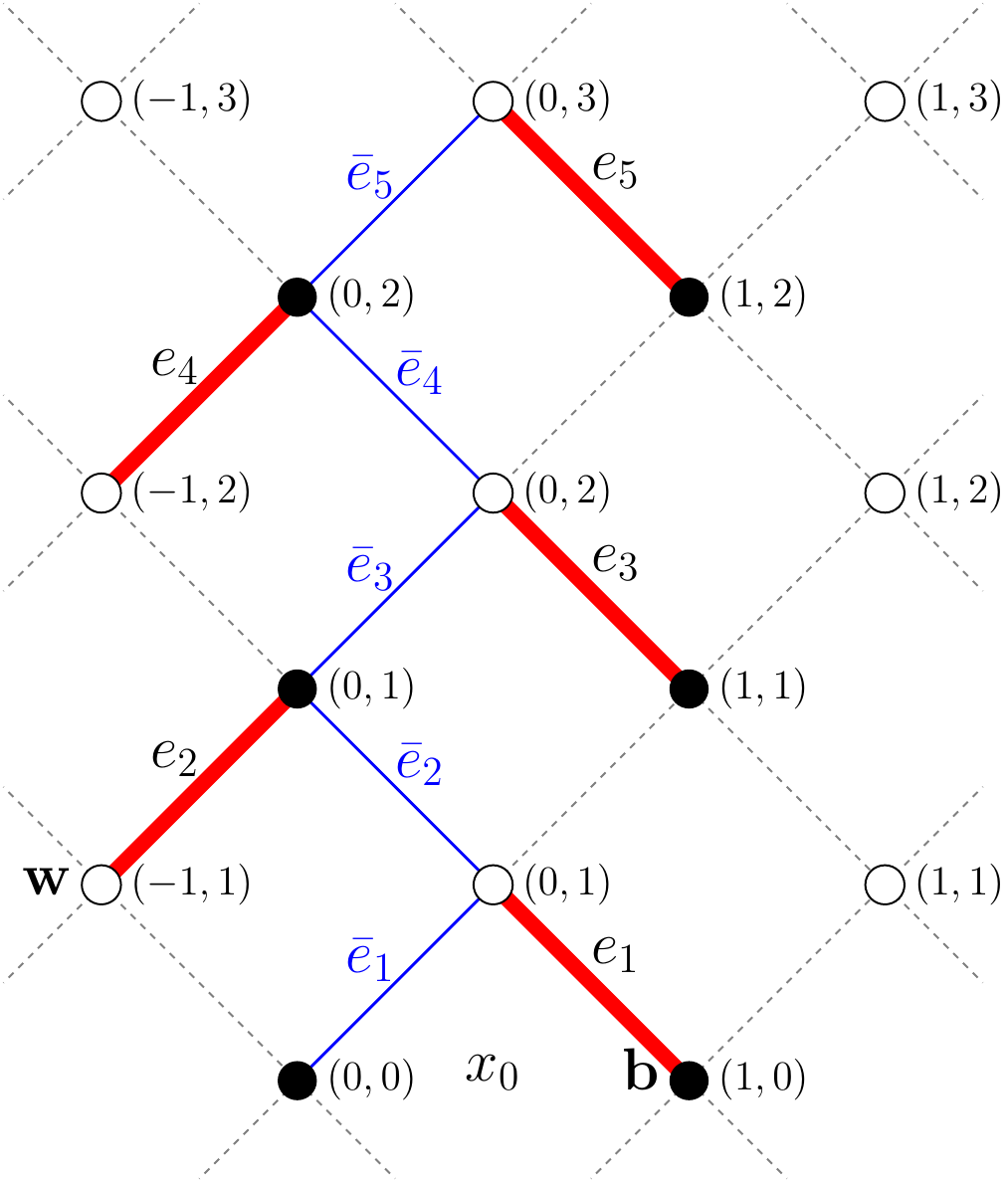}\hspace{5mm}
\includegraphics[height=8cm]{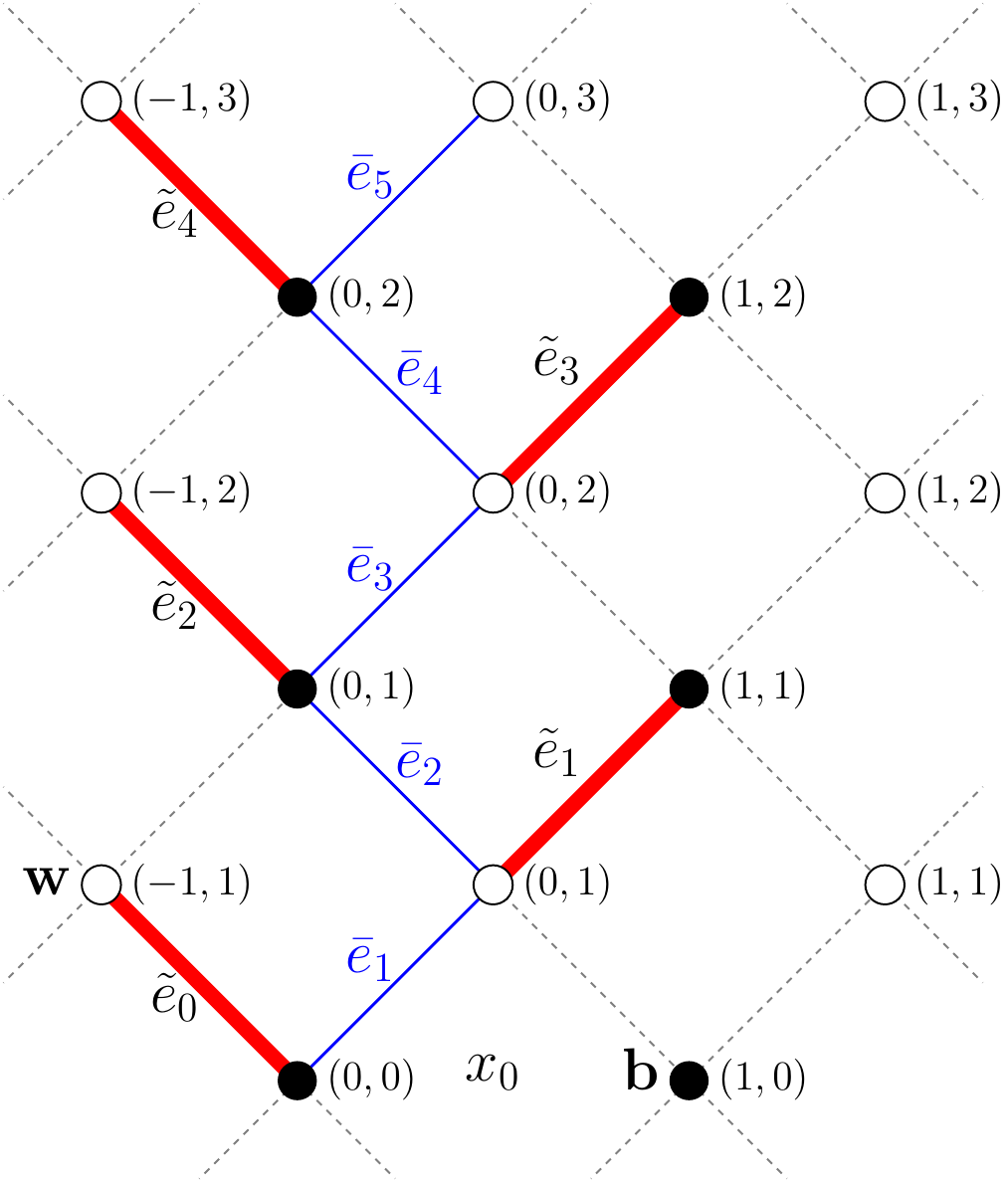}
\caption{The left figure shows the edges $e_1,e_2,\dots,$ etc., which are shown with solid red lines. The right figure shows the edges $\tilde{e}_0,\tilde{e}_1,\dots,$ etc., which are showed with solid red lines.
}
\label{fig:coords}
\end{center}
\end{figure}
Also, it is easily checked that the event $U(\bar e_1)$ is equivalent to the event that edge $e_1$ is occupied by
a dimer while $\tilde e_0$ is not. Finally, the event $U(\bar e_i)$, $i\ge 2$, is equivalent to the event that
edges $e_1,\dots,e_i$ are all occupied by dimers.

As a consequence,
\begin{equation}
 \label{vz22}
v^{\mathbb Z^2}(\rho)=\pi^{\mathbb Z^2}_\rho( e_1,\tilde e_0^c)+\sum_{n\ge 2}\pi^{\mathbb Z^2}_\rho(e_1,\dots,e_n).
\end{equation}
In Section~\ref{sec:vz2} we will show that the r.h.s.\ of \eqref{vz22} equals the r.h.s.\ of \eqref{vz2}.
The sum is convergent: in fact, label $x_i, i\ge0$ the faces in the
column $C_{\ell_0}$, where $x_i$ is adjacent to and above $x_{i-1}$. Then, the event
$\{e_1,\dots,e_n\in m\}$ is equivalent to $h(x_n)-h(x_0)=-n/2$. On the
other hand, $\pi_\rho^{\mathbb Z^2}(h(x_n)-h(x_0))=n\rho_2$
so that
\begin{equation}
\pi^{\mathbb Z^2}_\rho(e_1,\dots,e_n)=\pi^{\mathbb Z^2}_\rho\left[
h(x_n)-h(x_0)-\pi^{\mathbb Z^2}_\rho(h(x_n)-h(x_0))=-n(\rho_2+1/2)
\right].
\end{equation}
Observe that
$\rho_2>-1/2$ since $\rho\in P^{\mathbb Z^2}$. Finally, since the $k^{th}$ centered
moment of $h(x_n)-h(x_0)$ is $\Or((\log n)^{k/2})$ \cite[App. A]{laslier2015quickly} the summability of
\eqref{vz22} follows.

\subsubsection{Variance of the current}

\label{sec:vacu}
Let us move to the variance of $J(t)$ for $\mathcal G=\mathcal H$ (we do not work out formulas for $\mathcal G=\mathbb Z^2$). Recall that, given a dimer
configuration $m\in\mathcal M_{\mathcal H}$ and a horizontal edge $e$,
we denote $p(e)$ the highest particle below $e$ in the same
column. We denote $V(e)$ the number hexagonal faces that $p(e)$ has to
cross in order to reach edge $e$ and we set $V(e)=0$ if $p(e)$
cannot be moved to $e$ (i.e., if the move violates interlacements).

Going back to \cite[Sec. 9 and Appendix A]{Ton15}, one sees that to
prove Theorem~\ref{th:varianza} it is sufficient to show:
\begin{thm}
 \label{th:varV}
   Denote by $\Lambda_L$ the set of horizontal edges $e=(\bullet(x+1,n),\circ(x,n+1))$ with $0\le x\le L,0\le n\le L$. Then, for every $\rho\in P^{\mathcal H}$ there exists a constant $c $ such that
 \begin{equation}
 \label{eq:varV}
 {\rm Var}_{\pi_\rho}\bigg(\sum_{e\in\Lambda_L}V(e)\bigg)\le c L^2\log L.
 \end{equation}
\end{thm}

\subsection{Dimer coverings of bipartite graphs}

In the following, we will need more general bipartite graphs $G=(V,E)$ than just $\mathbb Z^2$ and $\mathcal H$.
To each of the edges $e\in E$, we assign a positive number called an
\emph{edge weight}. We denote the weight of the edge $e$ by
$\omega(e)$ with $\omega:E\to \mathbb{R}_{>0}$. We denote the set of
dimer coverings by $\mathcal{M}_G$ and, if the graph is finite, we denote the partition function by $Z_G$.
That is,
\begin{equation}
Z_G=\sum_{m \in \mathcal{M}_G} \prod_{e \in m} \omega(e).
\end{equation}
We define $\Pb_G$ to be the dimer model probability measure on the graph
$G$, that is for $m \in \mathcal{M}_G$, $\Pb_G(m)=\prod_{e \in
 m}\omega(e)/Z_G$.
\begin{defin}
\label{def:ZEV}
 Given a subset of edges $E_1\subset E$ and a subset of vertices $V_1\subset V$, we write $G\backslash
\{E_1,V_1\}$ to be the graph $G$ with all the edges in $E_1$ and
	vertices in $V_1$ removed from $G$, along with the edges incident to either $V_1$ or $E_1$. Let $Z_G [E_1,V_1]$ denote the partition
function of this graph (if either $E_1$ or $V_1$ is empty we omit it from the notation).
\end{defin}

 We use $ K_G$ to denote the \emph{Kasteleyn matrix} of $G$ which has columns indexed by the black vertices and rows indexed by white vertices with entries given by
\begin{equation}
	(K_G)_{{b}{w}} = \left\{ \begin{array}{ll}
	\mathrm{sign}(e	)\omega(e)	& \mbox{if $e=({b},{w})$ is an edge}, \\
	0 & \mbox{if ${w}$ and ${b}$ are not connected by an edge},
\end{array} \right.
\end{equation}
where $\mathrm{sign}(e)$ is a modulus-one complex number chosen so as
to satisfy the following property. Given a face $f$ of the graph, let
$e_1,\dots,e_{2n}$ be the edges incident to it, ordered say clockwise
with an arbitrary choice for $e_1$. Then, we impose that
\begin{equation}
\label{Phif}
\Phi(f):=\frac{\sgn(e_1)\sgn(e_3)\dots \sgn(e_{2n-1})}{\sgn(e_2)\sgn(e_4)\dots \sgn(e_{2n})}=(-1)^{n+1}.
\end{equation}
This is called a \emph{Kasteleyn orientation}. Existence of a
Kasteleyn orientation for every (bipartite) planar graph is known
\cite{Kas61} and in general many choices are possible. When $G$ is a
bipartite sub-graph of the infinite lattice $\mathcal G=\mathbb Z^2$ or
$\mathcal H$ and $e=(b,w)$, the restriction of $\bar K$ to $G$
\emph{does not} in general provide a correct Kasteleyn orientation for
$G$ and this will be an important point later.

Kasteleyn~\cite{Kas61,Kas63} and independently Temperley and
Fisher~\cite{TF61} noticed that $Z_G=|\det K_G|$ for domino tilings
 (to be more precise, their formulations involved the more
complicated non-bipartite graphs but the above formulation is
sufficient for this paper). This identity is true irrespective of the choice of Kasteleyn orientation and holds for any bipartite finite planar graph.
An observation due to Kenyon~\cite{Ken97}
shows that statistical properties can be found using the inverse of
the Kasteleyn matrix, that is, for
$e_1=(b_1,w_1),\dots, e_m=(b_m,w_m)$ edges in the graph $G$,
\begin{equation}\label{eq:localstats}
\Pb_G(e_1,\dots,e_m) = \bigg(\prod_{i=1}^m K(b_i,w_i)\bigg) \det[K^{-1}(w_i,b_j)]_{1 \leq i,j \leq m}.
\end{equation}
Actually, the analogous formula \eqref{eq:localstatsG} for the
infinite graph is obtained from \eqref{eq:localstats} by suitably
letting $G$ tend to the infinite graph $\mathcal G$ by toroidal
exhaustion \cite{KOS03}.

 \begin{rem}
\label{rem:gauge}
Given an edge weight function $\omega:e\in E\mapsto \omega(e)>0$,
 define face weights as the alternating product of the
 edge weights: given a face $f$ of $G$ adjacent to edges
 $e_1,\dots,e_{2n}$ (say in clockwise order with a given choice of $e_1$), let
\begin{equation}
\omega(f)=[\omega(e_1)\omega(e_3)\dots \omega(e_{2n-1})]/[\omega(e_2)\omega(e_4)\dots \omega(e_{2n})].
\end{equation}
 The dimer model probability measure is uniquely parametrized by its
 face weights, which means that two edge weight functions lead to the
 same probability measure if the corresponding face weights are
 equal. Suppose that $\omega_1$ and $\omega_2$ are two such edge
 weight functions. Then there exist functions $F_\circ$ and
 $F_\bullet$ on white and black vertices respectively such that
 $\omega_1 (e)/\omega_2 (e)=F_\circ(w)F_\bullet(b)$ for each edge
 $e=(w,b)$. We say that $\omega_1$ and $\omega_2$ are gauge equivalent
 and the act of multiplying edge weights by functions defined on its
 incident vertices is called a gauge transformation. The
 Kasteleyn matrix for a gauge equivalent weighting is obtained by
 pre-and post-composing with diagonal matrices built from the gauge
 transformation functions.

 \end{rem}

\section{Speed of growth on $\mathbb{Z}^2$}\label{sec:vz2}
In this section, we prove~\eqref{vz2}.
We begin by remarking that the edges $e_{2m+1},m\ge0 $ that appear in
formula \eqref{vz22} are the edges $(\bullet(1,m),\circ(0,m+1))$ while
$e_{2m},m\ge1$ are the edges $(\bullet(0,m),\circ(-1,m))$; see the
left picture in Figure~\ref{fig:coords}. We also remark that
$\tilde{e}_{2m+1}$ and $\tilde{e}_{2m}$ are the edges
$(\bullet(1,m+1),\circ(0,m+1))$ and $(\bullet(0,m),\circ(-1,m+1))$ for
$m\geq0$ respectively; see the right picture in
Figure~\ref{fig:coords}. Set $\Sigma_m=\{e_1,\dots, e_m\}$ and
$\tilde{\Sigma}_m=\{\tilde{e}_1,\dots, \tilde{e}_m\}$ with the
convention that $\tilde{\Sigma}_0=\emptyset$.

\subsection{Finite Graph}\label{sec:w1}

Consider a finite bipartite graph $G$ contained in $\mathbb{Z}^2$ and
set all edge weights  to $1$. Throughout this
section, we denote $\w$ to be the vertex $\circ(-1,1)$ and $\b$ to be
the vertex $\bullet(1,0)$.
See Figure~\ref{fig:sqcoords}.
\begin{figure}
\begin{center}
\includegraphics[height=6cm]{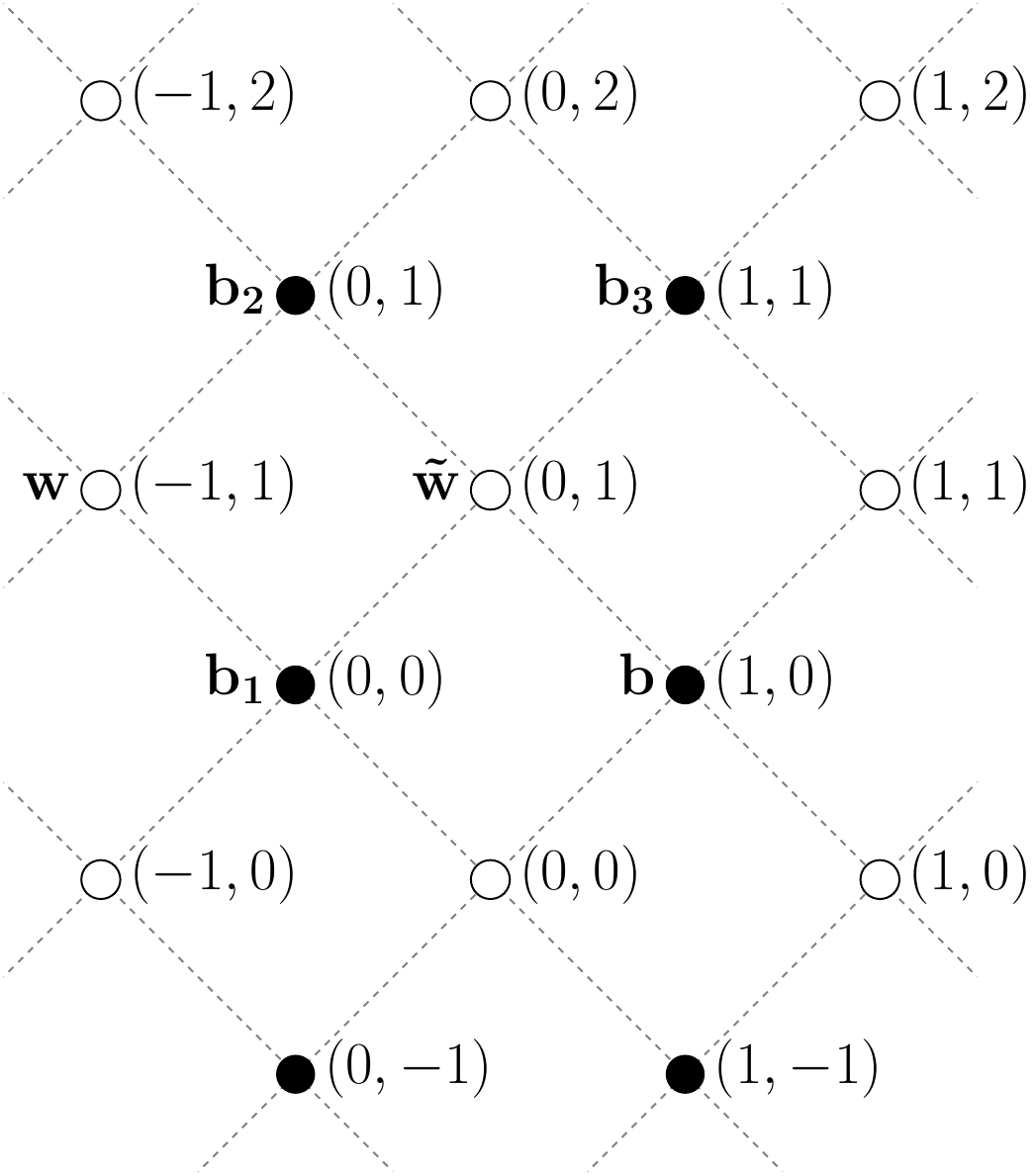}
\caption{The coordinate system used for square grid with the vertices ${\w}$, $\tilde{\w}$, $\b$, $\b_1$, $\b_2$ $\b_3$ and $\b_4$.}
\label{fig:sqcoords}
\end{center}
\end{figure}
With the notation of Definition~\ref{def:ZEV} we have
\begin{lem} \label{lem:finitegraphlemma}
For $l\geq 2$,
\begin{equation}
	\frac{Z_G[\{\w,\b\}]}{Z_G}=\Pb_G [\tilde{e}_0,e_1] +\sum_{k=2 }^{l} \Pb_G[ \Sigma_k] + R^{l}_G,
\end{equation}
where
\begin{equation}\label{RlG}
 R^l_G =\frac{ Z_G[ \tilde{\Sigma}_{l-1} ,\{\w,\b\}]}{Z_G}.
\end{equation}
and we assume that $G$ is large enough to include $\Sigma_l$ and $\tilde\Sigma_l$.
\end{lem}
The above lemma and its proof have a similar flavour to \cite[Proposition 3.5]{CF15} with the key difference that $\w$ and $\b$ are not on the same face.
\begin{proof}
 Consider the graph $G \backslash\{\w,\b\}$. There are three
 possibilities for the dimers incident to the vertex
 $ \tilde{\w}:=\circ(0,1)$. These are given by the edges
 $(\bullet(0,0),\circ(0,1))$, $(\bullet(0,1),\circ(0,1))$ and
 $(\bullet(1,1),\circ(0,1))=\tilde{e}_1$; see
 Figure~\ref{fig:firstmove}.
\begin{figure}[t]
\begin{center}
\includegraphics[width=\textwidth]{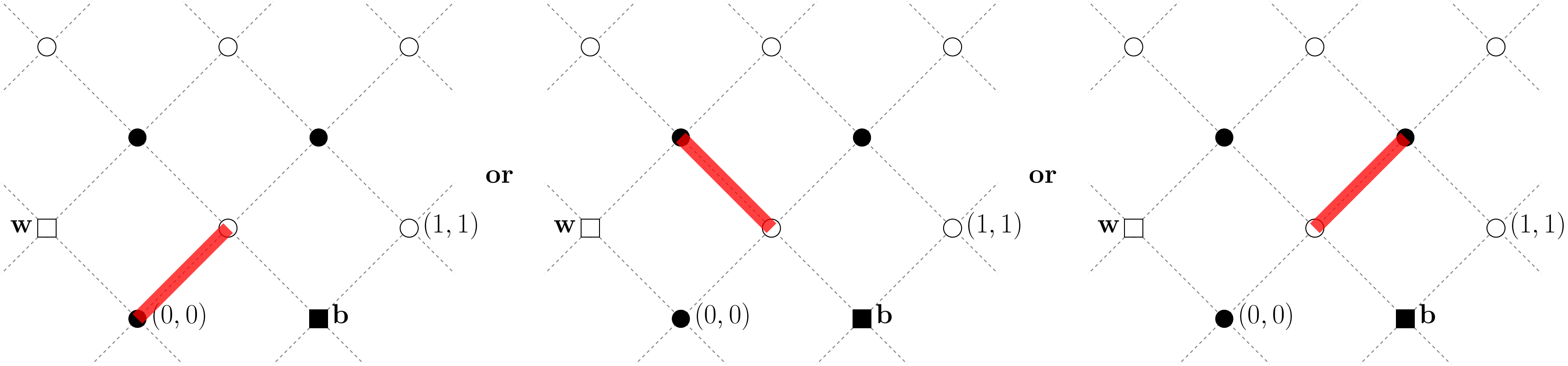}\hspace{5mm}
\caption{The three possible choices of dimers covering $\tilde{\w}$ on the graph $G\backslash \{\w,\b\}$. The vertices $\w$ and $\b$ are depicted by squares to stress that they have been removed from the graph, together with the edges incident to them.}
\label{fig:firstmove}
\end{center}
\end{figure}
If a dimer covers the edge
 $(\bullet(0,0),\circ(0,1))$, then the remaining graph is the same as
 $G \backslash \{\tilde{e}_0,e_1\}$. If a dimer covers the
 edge $(\bullet(0,1),\circ(0,1))$ instead, the remaining graph is the same as
 $G \backslash \Sigma_2$. This gives (remember that all edge weights equal $1$)
\begin{equation} \label{eq:removepoints}
Z_G[\{{\w,\b}\}] = Z_G[ \{\tilde{e}_0,e_1 \}] + Z_G[\Sigma_2] +  Z_G[ \tilde{\Sigma}_1,\{{\w,\b}\}],
\end{equation}
which can readily be seen from Figure~\ref{fig:firstmove}.

For $m \geq 1$, we have inductively the equations
\begin{equation}\label{eq:rec1}
Z_G[ \tilde{\Sigma}_{2m-1}, \{{\w,\b}\}]
	=  Z_G[ {\Sigma}_{2m+1}]
	+  Z_G[ \tilde{\Sigma}_{2m}, \{{\w,\b}\}],
\end{equation}
and
\begin{equation}\label{eq:rec2}
 Z_G[ \tilde{\Sigma}_{2m}, \{{\w,\b}\}]
	=  Z_G[ {\Sigma}_{2m+2}]
	+ Z_G[ \tilde{\Sigma}_{2m+1},\{ {\w,\b}\}].
\end{equation}
Indeed, \eqref{eq:rec1} follows because from the graph
${G \backslash( \tilde{\Sigma}_{2m-1}\cup \{\w,\b\})}$, there are two
possible dimers covering the vertex $\bullet(0,m)$: either the edge
$(\bullet(0,m),\circ(0,m+1))$ is covered by a dimer or the edge
$(\bullet(0,m),\circ(-1,m+1))=\tilde{e}_{2m}$ is covered by a dimer;
see Figure~\ref{fig:secondmove} for the case when $m=1$. Then,
\eqref{eq:rec1} follows after noticing that the graph
${G \backslash (\tilde{\Sigma}_{2m-1}\cup \{
  (\bullet(0,m),\circ(0,m+1))\} \cup \{\w,\b\}})$ is the same as the
graph ${G \backslash {\Sigma}_{2m+1}}$.
\begin{figure}
\begin{center}
\includegraphics[height=5cm]{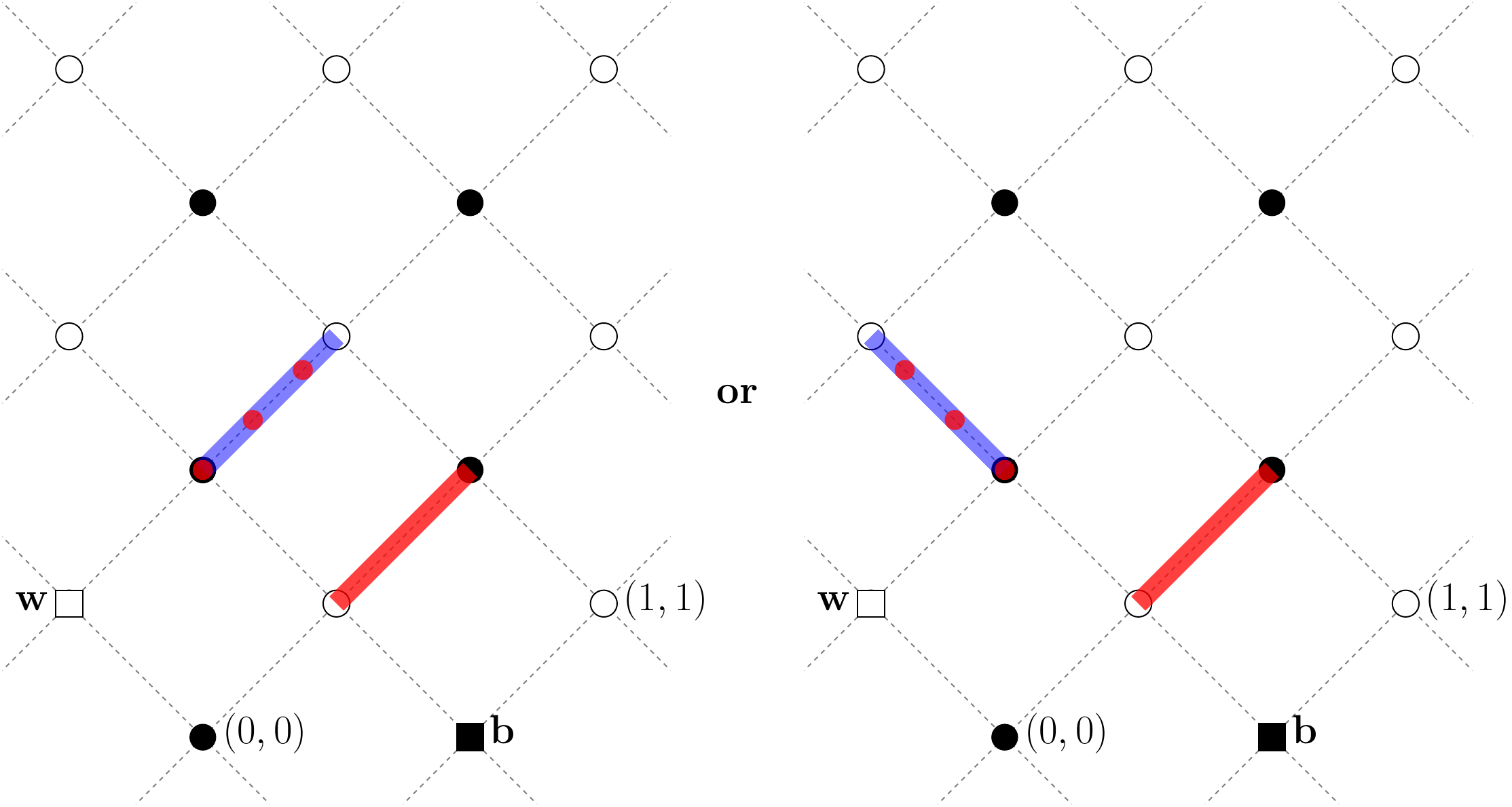}
\caption{There are two choices of dimers, drawn in blue with a dotted red line overlaid, covering the vertex $\circ(0,1)$ given that the dimer in blue is already present on $G\backslash \{\w,\b\}$. As in Figure~\ref{fig:firstmove}, the vertices $\w$ and $\b$ are depicted by squares to stress that they have been removed from the graph.
 The left figure leads to first term on the right side of~\eqref{eq:rec1} while the right figure leads to the second term on the right side of~\eqref{eq:rec1}.}
 \label{fig:secondmove}
\end{center}
\end{figure}

Similarly, to show \eqref{eq:rec2}, there are two possible dimers
covering the vertex $\circ(0,m+1)$ which are
$(\bullet(0,m+1),\circ(0,m+1))$ or
$(\bullet(1,m+1),\circ(0,m+1))=\tilde{e}_{2m+1}$. Then
\eqref{eq:rec2} follows after noticing that the graph
${G \backslash( \tilde{\Sigma}_{2m}\cup \{
 (\bullet(0,m+1),\circ(0,m+1))\} \cup \{\w,\b\})}$
is the same as the graph ${G \backslash {\Sigma}_{2m+2}}$. We
substitute the recursions in \eqref{eq:rec1} and
\eqref{eq:rec2} into \eqref{eq:removepoints} to give
\begin{equation}
Z_G[\{{\w,\b}\}]= Z_G[ \{\tilde{e}_0,e_1\} ] +  Z_G[ \tilde{\Sigma}_{l-1}, \{{\w,\b}\} ]
	+ \sum_{k=2}^l Z_G[ \Sigma_{k}].
\end{equation}
We divide the above equation by $Z_G$ and use the fact that
$\Pb_G(\Sigma_k)= Z_G[ \Sigma_k] /Z_G$. The claim is
proved.
\end{proof}

Recall that $K_G$ denotes the Kasteleyn matrix of $G$.
If we want a Kasteleyn matrix for $G\backslash\{\w,\b\}$, we cannot
just take the restriction of $K_G$. The problem is that the, since
$\Phi(f_i)=-1$ for the four $1\times 1$ square faces $f_1,\dots, f_4$
of $G$ around both $\w$ and $\b$ (recall \eqref{Phif} for the
definition of $\Phi(f)$), for the $2\times 2$ square faces
$f_{\w},f_{\b}$ that $G\backslash\{\w,\b\}$ has around $\w,\b$ we get
$\Phi(f_{\w})=\Phi(f_{\b})=(-1)^4=1$ which does not satisfy
\eqref{Phif}. This is easily fixed: to define a valid Kasteleyn
orientation on $G\backslash\{\w,\b\}$ we need to reverse the
orientation of a `path of edges' connecting the two faces
$f_{\w},f_{\b}$. In our case, as we explain below, it is sufficient to reverse the orientation of a single edge. A similar idea on a much more complicated scale was
used in great success in \cite{Dub15} to find correlations in the
monomer-dimer model.

We set $\tilde{\w}$ to be the vertex $\circ(0,1)$, $\b_1$ to be the
vertex $\bullet(0,0)$, $\b_2$ to be the vertex $\bullet(0,1)$ and
$\b_3$ to be the vertex $\bullet(1,1)$; see Figure~\ref{fig:sqcoords}.
For simplicity, we organize the matrix $K_G$ so that $\w$ and
$\tilde{\w}$ are in columns $1$ and $2$ while $\b_1$, $\b_2$, $\b_3$
and $\b$ are in rows $1$ to $4$ (in that order).
\begin{defin}
  \label{def:kgv}
  We let $K_G|_{V_1}$
 be the matrix obtained from the matrix $K_G$ by removing the rows
and columns associated to the black and white vertices from $V_1$
respectively, for some collection of vertices $V_1 \subset V$.
\end{defin}
Further, define $\tilde{K}_G=\tilde{K}_G({b,w})$ for
$w \in W_G\backslash \{\w\}$ and $b\in B_G \backslash\{\b \}$ by
\begin{equation} \label{eq:Ktilde}
\tilde{K}_G(b,w)= \left\{ \begin{array}{ll}
	K_G|_{\{\b,\w\}}({b,w}) & \mbox{if } (b,w)\not = (\b_1,\tilde{\w}), \\
	-K_G({\b_1,\tilde{\w}}) & \mbox{if } (b,w)= (\b_1,\tilde{\w}).
\end{array} \right.
\end{equation}
Observe that $\tilde K_G$ is a Kasteleyn matrix for $G\setminus\{\w,\b\}$, that satisfies \eqref{Phif}.

The following lemma relates $\frac{Z_G[ {\w,\b}]}{Z_G}$ with entries of the inverse of $K_G$.
\begin{lem} \label{lem:monomerformula}
It holds
\begin{equation}
\frac{Z_G[ \{{\w,\b}\}]}{Z_G}= \left|K_G^{-1}({\w,\b}) + \frac{2 K_G({\b_1, \tilde{\w}})}{ K_G({\b_1, w}) K_G({\b, \tilde{w}})} \Pb_G[\tilde{e}_0,e_1] \right|.
\end{equation}
\end{lem}
\begin{proof}
 The only nonzero entries of $\tilde{K}_G$ in the first column, which
 is indexed by the vertex $\tilde{\w}$, are the first three rows,
 which are indexed by the vertices $\b_1$, $\b_2$ and
 $\b_3$. Expanding out the determinant using the first column and
 noting~\eqref{eq:Ktilde} gives
	\begin{equation}
\begin{split}
 \det \tilde{K}_G
  =& -K_G({\b_1,\tilde{w}}) \det K_G|_{\{\b,\w,\b_1,\tilde{w}\}} -K_G({\b_2,\tilde{w}}) \det K_G|_{\{\b,\w,\b_2,\tilde{w}\}}\\
 &+K_G({\b_3,\tilde{w}}) \det K_G|_{\{\b,\w,\b_3,\tilde{w}\}}.
\end{split}
\end{equation}
Similarly, the matrix ${K}_G|_{\{\b,\w\}}$, which corresponds to removing
$\b$ and $\w$ from $K_G$, has the same nonzero entries in the first
column as the matrix $\tilde{K}_G$. Expanding the determinant by the
first column of ${K}_G|_{\b,\w}$ gives
\begin{equation}
\begin{split}
\det {K}_G|_{\{\b,\w\}}
	 =&K_G({\b_1,\tilde{w}}) \det K_G|_{\{\b,\w,\b_1,\tilde{w}\}} -K_G({\b_2,\tilde{w}}) \det K_G|_{\{\b,\w,\b_2,\tilde{w}\}}\\
	&+K_G({\b_3,\tilde{w}}) \det K_G|_{\{\b,\w,\b_3,\tilde{w}\}}.
\end{split}
\end{equation}
Using the above two equations, we conclude that
	\begin{equation}
\det \tilde{K}_G = \det {K}_G|_{\{\b,\w\}}-2K_G({\b_1,\tilde{w}}) \det K_G|_{\{\b,\w,\b_1,\tilde{w}\}}.
\end{equation}
We divide both sides of the above equation by $\det K_G$ and take absolute values of both sides which gives
\begin{equation}
 \frac{Z_G[\{ {\w,\b}\}]}{Z_G} = \left|\frac{ \det \tilde{K}_G}{{\det K_G}} \right| =\left| \frac{ \det {K}_G|_{\{\b,\w\}}}{\det K_G}-2K_G({\b_1,\tilde{w}}) \frac{ \det K_G|_{\{\b,\w,\b_1,\tilde{w}\}}}{\det K_G} \right|,
\end{equation}
where we recall that
$Z_G=|\det K_G|,Z_G[\{ {\w,\b}\}]=|\det \tilde K_G|$ by Kasteleyn's
Theorem \cite{Kas61}. The claim then follows since
$\det {K}_G|_{\{\b,\w\}}/{\det K_G}=-K_G^{-1}({\w,\b})$ (recall that $\w$ and
$\b$ are in the first column and fourth row of $K_G$) and due to
\begin{equation}
\begin{aligned}
&2K_G({\b_1,\tilde{w}}) \frac{\det K_G|_{\{\b,\w,\b_1,\tilde{w}\}} }{\det K_G}	\\
=&\frac{2 K_G({\b_1, \tilde{\w}})}{ K_G({\b_1, w}) K_G({\b ,\tilde{w}})} { K_G({\b_1, w}) K_G({\b, \tilde{w}})} \frac{ \det K_G|_{\{\b,\w,\b_1,\tilde{w}\}} }{\det K_G} \\
=& \frac{2 K_G({\b_1, \tilde{\w}})}{ K_G({\b_1, w}) K_G({\b ,\tilde{w}})} \Pb_G[\tilde{e}_0,e_1],
\end{aligned}
\end{equation}
because $|\det K_G| =Z_G$ and the overall signs match up for $ \det K_G|_{\{\b,\w,\b_1,\tilde{w}\}} $ and ${\det K_G}$ \cite{Kas61}.
\end{proof}
The following corollary follows immediately from the statements of Lemmas~\ref{lem:finitegraphlemma} and~\ref{lem:monomerformula}.
\begin{cor} \label{cor:Ktospeed}
We have
\begin{equation}
\left| K_G^{-1}({\w,\b}) + \frac{2 K_G({\b_1, \tilde{\w}})}{ K_G({\b_1, w}) K_G({\b, \tilde{w}})} \Pb_G[\tilde{e}_0,e_1] \right|
= \Pb_G [\tilde{e}_0,e_1]+\sum_{k=2 }^{l} \Pb_G[\Sigma_k]
	 + {R^{l}_{G}},
\end{equation}
where $R^l_G$ is given in~\eqref{RlG}.
\end{cor}

\begin{lem} \label{lem:RlGsquare}
We have
\begin{equation}
0\le R^l_G \leq \Pb_G[e_1,\dots,e_{l-1}].
\end{equation}
\end{lem}

\begin{proof}The lower bound is obvious since $R^l_G$ is the ratio of
  two partition functions.  We give the proof of the upper bound when
  $l-1=r$ is even; a similar argument holds for $r$ odd.

Notice that the set of edges incident to
$\tilde{\Sigma}_{r} \cup \{{ \w} ,{\b } \}$ equals the set of edges
incident to $\Sigma_r \cup \{ \circ(-1, r/2+1) ,\bullet (1,r/2) \}$
and so we have that
\begin{equation}
\begin{aligned}
 Z_G[ \tilde{\Sigma}_{r} , \{{\w ,\b}\}]&=Z_G[ \Sigma_r , \{{ \circ(-1, r/2+1) ,\bullet (1,r/2)\} }]\\
 &=Z_{G\backslash \Sigma_r}[\{ { \circ(-1, r/2+1) ,\bullet (1,r/2) }\}],
\end{aligned}
\end{equation}
where as usual $G\backslash \Sigma_r$ stands for the graph $G$ with
all incident edges to $\Sigma_r$ removed, that is
$Z_{G\backslash \Sigma_r}=Z_G[\Sigma_r]$. On the graph
$G\backslash \Sigma_r$, the vertices $\circ(-1, r/2+1)$ and
$\bullet (1,r/2) \}$ are on the same face. As is easily checked, this
means that removing these vertices does not change the overall
Kasteleyn orientation from $G \backslash \Sigma_r$. Hence, we have
(recall the notation $K_G|_{V_1}$ from Definition \ref{def:kgv})
\begin{equation}
 \frac{Z_{G \backslash \Sigma_r} [ \{{ \circ(-1, r/2+1) ,\bullet (1,r/2)}\}]}
{Z_{G \backslash \Sigma_r}} = \left|\frac{\det( K_{G\backslash \Sigma_r}|_{\{\bullet (1,r/2),\circ(-1, r/2+1) \}})}{\det( K_{G\backslash \Sigma_r})} \right| \leq 1.
\end{equation}
The inequality holds because each term in the expansion of the
determinant in the numerator is also present in the denominator. By noting that the denominator could have
more terms and since all the terms in the expansion of the
determinants have the same sign by Kasteleyn's theorem, the inequality follows. Multiplying
both sides of the above inequality by $Z_{G \backslash \Sigma_r}$ and
dividing both sides by $Z_G$ gives the result.
\end{proof}

\subsection{Infinite volume limit and proof of \eqref{vz2}}
\label{sec:varianza}
In this section, we extend the formula in Corollary~\ref{cor:Ktospeed}
to the infinite volume limit:
\begin{prop}\label{prop:squareinfinite}
Let $\bar K$ be the Kasteleyn matrix of $\mathbb Z^2$ defined in Section~\ref{subsec:Gibbs} and let $B_1,B_2$ be
related to the slope $\rho$ by \eqref{eq:Zslope1} and \eqref{eq:Zslope2}. Then,
\begin{equation}
  \begin{split} \label{speedinfinitevol}
    &\frac {e^{B_2}}{e^{B_1}}\left| \bar{K}^{-1}({\w},{\b}) +
      \frac{2 \bar K({\b_1} ,\tilde{\w})}{ \bar K({\b_1}, {\w}) \bar K({\b} ,\tilde{\w})} \pi^{\mathbb Z^2}_\rho[\tilde{e}_0,e_1] \right|
    = \pi^{\mathbb Z^2}_\rho [\tilde{e}_0,e_1] +\sum_{k=2 }^{\infty} \pi^{\mathbb Z^2}_\rho [ e_1,\dots,e_k].
\end{split}
\end{equation}
\end{prop}
Note that the sum that appears in the right side is the same as in the
definition of average current, \eqref{vz22}.

\begin{proof}[Proof of Proposition~\ref{prop:squareinfinite}]
 We start from Corollary~\ref{cor:Ktospeed} and we take a graph $G$
 that tends to $\mathbb Z^2$ in such a way that around the vertices
 $\w,\b$ the dimer statistics $\mathbb P_G$ tends to that of
 $\pi^{\mathbb Z^2}_\rho$. Our choice for $G$ is a suitable space translation of the so-called Aztec
 diamond: using the same
 coordinate system as above, the Aztec diamond $A_L$ is a $L\times L$ subset of $\mathbb{Z}^2$ whose
 white vertices are given by
 $ W_{A_L}:=\{ \circ (x,y): 0 \leq x \leq L-1,0 \leq y \leq L\}$, black
 vertices given by
 $B_{A_L}:=\{\bullet (x,y): 0\leq x \leq L,0\leq y \leq L-1\}$ and whose edge set contains all the edges connecting $W_{A_L}$ to $B_{A_L}$.
 As in Section~\ref{sec:w1} we assign
 weight $1$ to all the edges of $A_L$. The Kasteleyn orientation we choose is $K_{A_L}(b,w)=1$ for every vertical edge and
$K_{A_L}(b,w)=i$ for every horizontal edge.

 For uniformly random domino tilings of Aztec diamonds, the local
 behavior of the tiling separates, as $L\to \infty$, into two
 distinct macroscopic regions and the interface between these two
 regions is referred to as the \emph{limit shape} or \emph{limit
 shape curve}. See~\cite{KenLectures} for a more complete
 overview: here we recall only what we need for our present work. Rescaling the Aztec diamond by $L$, so that the corners
 are given by $(0,0), (1,0), (1,1)$ and $(0,1)$, the limit shape
 is given by a circle of radius $1/2$ whose center is at $(1/2,1/2)$.
 We denote the open disk inside the circle by $\mathcal{D}$.

Fix $\xi=(\xi_1,\xi_2)\in \mathcal D$ and let $G_L$ be the Aztec diamond $A_L$ translated by $(-\lfloor \xi_1 L\rfloor,
-\lfloor \xi_2 L\rfloor)$. Then, \cite[Theorem 2.9]{CJY15} says that for any local dimer observable $f$, one has the convergence
\begin{equation}
\label{cb}
\lim_{L\to\infty} \mathbb P_{G_L}(f)=\pi^{\mathbb Z^2}_{\hat\rho(\xi)}(f),
\end{equation}
where $\hat\rho(\xi):=\rho(B(\xi))$, with
$\rho(\cdot)$ as in \eqref{eq:Zslope1}-\eqref{eq:Zslope2} and
$B=B(\xi)=(B_1(\xi),B_2(\xi))$ given by
\begin{equation}
 \label{eq:biro}
B_i(\xi)= 1/2 \log (\xi_i/(1-\xi_i)).
\end{equation}
Moreover, the inverse Kasteleyn matrix $K^{-1}_{A_L}$ satisfies, for every fixed pairs of vertices $(\circ(x_1,x_2), \bullet (y_1,y_2))$,

\begin{equation}
\label{eq:gdn}
	\lim_{L\to\infty}K^{-1}_{G_L}(\circ(x_1,x_2), \bullet (y_1,y_2))= e^{B_1(y_1-x_1-1)} e^{B_2(y_2-x_2)}
	\overline{K}^{-1}(\circ(x_1,x_2), \bullet (y_1,y_2)),
\end{equation}
To understand the exponential factor in the above formula, first
notice that the left side of~\eqref{eq:gdn} is the inverse Kasteleyn matrix corresponding to a Kasteleyn weighting of
$\mathbb{Z}^2$ with weights equal to $1$ and $\mathrm{i}$, while $\overline K^{-1}$ in the
right side of~\eqref{eq:gdn} corresponds to $\mathbb{Z}^2$ having
weights described by the Kasteleyn matrix in~\eqref{eq:KZ}. The
measures on each of these graphs are gauge equivalent (in the sense of Remark~\ref{rem:gauge}) as there is a
gauge transformation from the graph corresponding to the left side
of~\eqref{eq:gdn} to the graph corresponding to the right side
of~\eqref{eq:gdn}. More explicitly, this is given by multiplying the
vertices $\bullet(y_1,y_2)$ by $e^{-B_2y_2-B_1y_1}$ and the vertices
$\circ(x_1,x_2)$ by $e^{B_1(x_1+1)+B_2x_2}$. This explains the
prefactor on the right side of~\eqref{eq:gdn}. The convergence of
$\lim_{L\to\infty}K^{-1}_{G_L}(\circ(x_1,x_2), \bullet (y_1,y_2))$ to
its full plane counterpart is given in the proof of \cite[Theorem
2.9]{CJY15}; see also Remark~\ref{rem:kastconv}.
			
We will see in a moment that it is always possible to find
$\xi\in \mathcal D$ such that $\hat\rho(\xi)$ equals the slope
$\rho\in P^{\mathbb Z^2}$ that appears in \eqref{speedinfinitevol}.
We have now all necessary ingredients to prove
\eqref{speedinfinitevol}. We start from
Corollary~\ref{cor:Ktospeed} with $G=G_L$. The probabilities
$\mathbb P_{G_L}[\dots]$ tend as $L\to\infty$ to the corresponding
$\pi_\rho^{\mathbb Z^2}$ probabilities by \eqref{cb}. The matrix element
$K_{G_L}^{-1}(\w,\b)$ tends, by \eqref{eq:gdn}, to
\begin{equation}
e^{B_2-B_1}\overline{K}^{-1}(\w,\b),
\end{equation}
while
\begin{equation}
 \frac{K_{G_L}(\b_1,\tilde{\w})}{K_{G_L}({\b_1},{\w})K_{G_L}({\b},\tilde{\w})}=-1=e^{B_1-B_2}
 \frac{\overline K(\b_1,\tilde{\w})}{\overline
 K({\b_1},{\w})\overline K(\b,\tilde{\w})}
\end{equation}
(recall \eqref{eq:KZ} and the choice of Kasteleyn matrix for the Aztec diamond, which is just as in \eqref{eq:KZ} with $B_1=B_2=0$).
Finally, by Lemma~\ref{lem:RlGsquare} we see that
\begin{equation}
0\le \limsup_{L\to\infty} R^l_{G_L}\le \pi^{\mathbb Z^2}_\rho[e_1,\dots,e_{l-1}],
\end{equation}
so that, letting $l\to\infty$ we obtain \eqref{speedinfinitevol} (we have already remarked in Section~\ref{sec:avc} that the series is convergent).

It remains only to prove that the image of the map
$\xi\in\mathcal D\mapsto \hat\rho(\xi)$ is the whole open square
$P^{\mathbb Z^2}$. In fact, it is easy to verify that
the map $\xi\mapsto (B_1,B_2),B_i=1/2 \log (\xi_i/(1-\xi_i))$ gives a
one-to-one correspondence between $\mathcal D$ and the amoeba
$\mathcal B$ defined in \eqref{eq:ameba} and we already mentioned that the map
$B\in \mathcal B\mapsto \rho(B)\in P^{\mathbb Z^2}$ in
\eqref{eq:Zslope1}-\eqref{eq:Zslope2} is also a bijection.
\end{proof}

\begin{rem}\label{rem:kastconv}
	For simplicity, the weights on $A_L$ were chosen to be  $1$ and $\mathrm{i}$.
	The Kasteleyn matrix for the Aztec diamond in the uniform case in~\cite{CJY15} differs by $K_{A_L}$ only up to sign, which means entries of the inverse differ up to a sign.

	The proof of \cite[Theorem 2.9]{CJY15} involves showing the
 convergence of the entries of $K_{G_L}^{-1}$ as $L$ tends to
 infinity. Note that this limiting inverse Kasteleyn matrix is
 an inverse of a Kasteleyn matrix different from the one we
 considered in this paper given in~\eqref{eq:KZ}; the two are gauge
 equivalent. We believe that the choice in this
 paper is more natural and aesthetically pleasing, mainly
 because the slopes are embedded into the edge weights, which
 mirrors the honeycomb case.
\end{rem}

\begin{proof}[Proof of Theorem~\ref{thm:square}]
We now compute the speed of growth for dynamics on $\mathbb{Z}^2$. Recalling formula \eqref{vz22} for the speed and Proposition~\ref{prop:squareinfinite}, we see that
\begin{equation}
v^{\mathbb{Z}^2}(\rho)
=\frac{e^{B_2}}{e^{B_1}}\left| \bar{K}^{-1}({\w},{\b}) + \frac{2 \bar K({\b_1} ,\tilde{\w})}{ \bar K({\b_1}, {\w}) \bar K({\b} ,\tilde{\w})} \pi^{\mathbb Z^2}_\rho[\tilde{e}_0,e_1] \right|	 - \pi^{\mathbb Z^2}_\rho[\tilde{e}_0,e_1]+ \pi^{\mathbb Z^2}_\rho[\tilde{e}_0^c,e_1],
\end{equation}
where $\tilde{e}_0^c$ is the event that the edge $\tilde{e}_0$ is not
present. The result \eqref{vz2} then follows immediately from
Lemma~\ref{lem:positivity} in Appendix~\ref{App:Positivity}.
\end{proof}

\section{Large time height fluctuations on ${\cal H}$}\label{sectThmVar}

\begin{rem}
  \label{rem:symm}
  The Gibbs measure $\pi_{\rho}^{\mathcal H}$ is invariant under
  translations and reflection through the center of any hexagonal
  face. In fact, such transformations clearly preserve the Gibbs
  property (the measure is locally uniform, conditioned on the
  configuration outside any finite domain) and leave the three dimer
  densities unchanged.  Given that the Gibbs measure with given
  densities is unique, the claim follows. Note that, under reflection,
  the function $V(e)$ transforms into $\widehat V(e')$ for some $e'$
  that depends on the face chosen as center of reflection. Here,
  $\widehat V(e)$ is the number of hexagonal faces that the
  \emph{lowest} horizontal dimer \emph{above} $e$ has to cross in
  order to reach $e$ ($\widehat V(e)=0$ if the move is not allowed).
\end{rem}
Recall that Theorem~\ref{th:varianza} follows by proving the
equilibrium estimate \eqref{eq:varV}.
For $i \in\{0,1\}$, let
$\Lambda_L^i$ denote the set of horizontal edges
$e=(\bullet(x+1,n),\circ(x,n+1))\in\Lambda_L$ with $x\!\!\!\mod 2=i$,
i.e., those in even (for $i=0$) or odd (for $i=1$) columns. By
Cauchy-Schwarz and Remark \ref{rem:symm} we have
\begin{equation}\label{eq4.2}
\begin{split}
\mathrm{Var}_{\pi_{\rho}^{\mathcal{H}}}\bigg(\sum_{e \in \Lambda_L} V(e)\bigg)
&=\mathrm{Var}_{\pi_{\rho}^{\mathcal{H}}}\bigg(\sum_{e \in \Lambda_L^0} \widehat V(e)+\sum_{e \in \Lambda_L^1}\widehat V(e)\bigg) \\
& \leq 2\mathrm{Var}_{\pi_{\rho}^{\mathcal{H}}}\bigg(\sum_{e \in \Lambda_L^0} \widehat V(e)\bigg)+2\mathrm{Var}_{\pi_{\rho}^{\mathcal{H}}}\bigg(\sum_{e \in \Lambda_L^1} \widehat V(e)\bigg).
\end{split}
\end{equation}

\begin{prop} \label{prop:evenoddest}
For $i\in\{0,1\}$, we have for some constant $C(\rho)>0$
\begin{equation}
 \label{eq:CSh}
 \mathrm{Var}_{\pi_{\rho}^{\mathcal{H}}}\bigg(\sum_{e \in \Lambda_L^i} \widehat V(e)\bigg)
=\sum_{e_1,e_2 \in \Lambda_L^i} \pi_{\rho}^{\mathcal{H}}(\widehat V(e_1);\widehat V(e_2))
 \leq C L^2 \log L,
\end{equation}
with $\pi(f;g):=\pi(f g)-\pi(f)\pi(g)$ (the covariance of $f$ and $g$).
\end{prop}
The proof is given in Section~\ref{subsec:estimates}.
The advantage of the decomposition \eqref{eq4.2} is that in \eqref{eq:CSh}
terms with $e_1,e_2$ in neighboring columns, that would require a special treatment, do not appear.
In most figures of this section we find it convenient to deform the
hexagonal faces of $\mathcal H$ into rectangles, as in the drawing on
the right of Figure~\ref{fig:hex}, so that the axes $\hat e_1,\hat e_2$ become orthogonal.

Given the horizontal edge $e=(\bullet(x+1,n),\circ(x,n+1))$, define the edge set
\begin{equation} \label{eq:Ome}
O_{m,e}= \bigcup_{i=0}^{m-1}\{ (\bullet(x,n+i+1),\circ(x,n+1+i)),(\bullet(x+1,n+i),\circ(x+1,n+i+1)) \}
\end{equation}
and
\begin{equation}
\tilde{O}_{m,e}=
\bigcup_{i=1}^{m}\{(\bullet(x,n+1-i),\circ(x,n+1-i)),(\bullet(x+1,n-i),\circ(x+1,n+1-i))\}
\end{equation}
for $m\geq 1$. Using the notation $e+m=(\bullet(x+1,n+m),\circ(x,n+1+m))$, $m\in\Z$, we have $O_{m,e}=\widetilde O_{m,e+m}$. Also, we define
\begin{equation}\label{eq:exmBB}
\widetilde V(e)=\sum_{m\geq 1} \Id_{\widetilde O_{m,e}}.
\end{equation}

\subsection{Expressions for ${\pi_{\rho}^{\mathcal{H}}}(\widetilde V(e_1) ;\widetilde V(e_2))$}
We first determine a more explicit expression for $\widehat V(e)$, which was defined in Remark~\ref{rem:symm}. We
use the notation that $\Id_{O_{m,e}}$ means the indicator event of
dimers covering the edges $O_{m,e}$. Setting $B_{m,e}=O_{m,e} \cup \{ (\bullet(x+1,n+m),\circ(x,n+m+1))\}$, and by considering the possible
dimers incident to the vertex $\bullet(x+1,n+m)$, we have
\begin{equation}
	\Id_{O_{m,e}}= \Id_{B_{m,e}}+\Id_{O_{m+1,e}}.
\end{equation}
By definition of $\widehat V(e)$ and the above equation
\begin{equation}
 \label{eq:exm}
	\widehat V(e)=\sum_{m=1}^{\infty} m \Id_{B_{m,e}} = \sum_{m=1}^{\infty} m(\Id_{O_{m,e}}-\Id_{O_{m+1,e}})=\sum_{m=1}^{\infty} \Id_{O_{m,e}}.
\end{equation}

By linearity and translation invariance, ${\pi_{\rho}^{\mathcal{H}}}[\widehat V(e)] ={\pi_{\rho}^{\mathcal{H}}}[\widetilde V(e)]$. As shown in~\cite{CF15}, the expectation of $\widetilde V(e)$ can be written in
terms of a single entry of $\overline{K}^{-1}$, namely
\begin{equation}\label{eqSpeedKernel}
{\pi_{\rho}^{\mathcal{H}}}[\widetilde V(e)] = -\frac{a_2a_3}{a_1} \overline{K}^{-1} (\circ(x+1,n),\bullet (x,n)).
\end{equation}
Extending the ideas of~\cite{CF15}, in
Proposition~\ref{prop:hex:infinite} we derive a formula for a
$2\times 2$ determinant of $\overline{K}^{-1}$ in terms of
$\widetilde O_{m,e}$. This will be almost the same as
${\pi_{\rho}^{\mathcal{H}}}(\widetilde V(e_1) ;\widetilde
V(e_2))$. Then, in Section \ref{subsec:estimates}, we will express the
variance \eqref{eq:CSh} in terms of correlations
${\pi_{\rho}^{\mathcal{H}}}(\widetilde V(e_1) ;\widetilde V(e_2))$.

\begin{prop} \label{prop:hex:variance}
For $j\in\{1,2\}$  consider the horizontal edges $e_j=(\bullet(x_j+1,n_j),\circ(x_j,n_j+1))$  with $x_1,x_2,n_1,n_2 \in \mathbb{Z}$.\\
If $|x_1-x_2|>1$, then
\begin{equation}\label{eq:speedvariance1}
{\pi_{\rho}^{\mathcal{H}}}[\widetilde V(e_1)\widetilde V(e_2)] = \sum_{m_1,m_2=1}^{\infty} \pi_\rho^{\mathcal{H}} [\widetilde O_{m_1,e_1}\,\widetilde O_{m_2,e_2}].
\end{equation}
If $x_1=x_2$ and $n_1>n_2$, then
\begin{equation}\label{eq:speedvariance2}
{\pi_{\rho}^{\mathcal{H}}}[\widetilde V(e_1)\widetilde V(e_2)] = \sum_{m_1=1}^{|n_1-n_2|-1}\sum_{m_2=1}^\infty\pi_\rho^{\mathcal{H}} [\widetilde O_{m_1,e_1}\,\widetilde O_{m_2,e_2}]+\sum_{m=1}^{\infty} 2m\, \pi_\rho^{\mathcal{H}} [\widetilde O_{|n_1-n_2|+m,e_1}].
\end{equation}
The case $x_1=x_2$ and $n_2<n_1$ is obtained by symmetry. Finally, if $e_1=e_2$, then
\begin{equation}\label{eq:speedvariance3}
{\pi_{\rho}^{\mathcal{H}}}[\widetilde V(e_1)\widetilde V(e_2)] = \sum_{m=1}^{\infty}(2m-1) \pi_\rho^{\mathcal{H}} [\widetilde O_{m,e_1}].
\end{equation}
\end{prop}
Convergence of the sums is shown later.

\begin{proof}
The statement for $|x_1-x_2|>1$ simply follows from \eqref{eq:exmBB}.

For $e_1=e_2$, using $\Id_{\widetilde O_{m_1,e_1}}\Id_{\widetilde O_{m_2,e_1}}=\Id_{\widetilde O_{\max\{m_1,m_2\},e_1}}$, we get
\begin{equation}
\widetilde V(e_1)\widetilde V(e_2)  = \sum_{m=1}^\infty \Id_{\widetilde O_{m,e_1}}+2\sum_{m_2=2}^\infty\sum_{m_1=1}^{m_2-1} \Id_{\widetilde O_{m_2,e_1}} = \sum_{m=1}^\infty (2m-1)\Id_{\widetilde O_{m,e_1}}.
\end{equation}

For $x_1=x_2$, we suppose that $n_1>n_2$. The result for $n_1<n_2$ is recovered by relabeling. We have
\begin{equation}\label{eq4.12}
\widetilde V(e_1)\widetilde V(e_2)=\sum_{m_1=1}^{|n_1-n_2|-1}\sum_{m_2=1}^\infty \Id_{\widetilde O_{m_1,e_1}} \Id_{\widetilde O_{m_2,e_2}} + \sum_{m_1=|n_1-n_2|}^{\infty}\sum_{m_2=1}^\infty \Id_{\widetilde O_{m_1,e_1}} \Id_{\widetilde O_{m_2,e_2}}.
\end{equation}
For the last term, the two $\widetilde O$ join so that $\Id_{\widetilde O_{m_1,e_1}} \Id_{\widetilde O_{m_2,e_2}} = \Id_{\widetilde O_{\max\{m_1,m_2+|n_1-n_2|\},e_1}}$. Using this, the second term in (\ref{eq4.12}) becomes
\begin{equation}
\sum_{m=1}^\infty 2m \Id_{\widetilde O_{m+|n_1-n_2|,e_1}}.
\end{equation}
Taking expectations with respect to ${\pi_{\rho}^{\mathcal{H}}}$ finishes the proof.
\end{proof}

\subsection{Expressions involving $\overline{K}^{-1}$}\label{subsectK1}
Recall that $\overline{K}^{-1}$ represents the inverse Kasteleyn matrix on $\mathcal{H}$ whose entries are given by~\eqref{hexagonK-1}.

\begin{prop}\label{prop:hex:infinite}
  Let $e_i=(\bullet(x_i+1,n_i),\circ(x_i,n_i+1))$ for  $i=1,2$.
  Then,
\begin{equation}
\begin{split}\label{eq:covarianceK:inf}
&\frac{(a_2a_3)^2}{a_1^2}\det \left( \overline{K}^{-1}(\circ(x_i+1,n_i),\bullet(x_j,n_j)) \right)_{1 \leq i,j \leq 2} \\
&= \left\{ \begin{array}{ll}
\displaystyle	\sum_{m_1=1}^{\infty} \sum_{m_2=1}^{\infty} \pi_\rho^{\mathcal{H}} [ \tilde{O}_{m_1,e_1} \tilde{O}_{m_2,e_2}] &\mbox{if $|x_1-x_2|>1$},\\
	\displaystyle \sum_{m_1=1}^{|n_1-n_2|-1} \sum_{m_2=1}^{\infty} \pi_\rho^{\mathcal{H}}[ \tilde{O}_{m_1,e_1} \tilde{O}_{m_2,e_2}]
+	\displaystyle\sum_{m=1}^{\infty}\pi_\rho^{\mathcal{H}}[\tilde{O}_{m+|n_1-n_2|,e_1}] & \mbox{if $x_1=x_2$, $n_1>n_2$}.
\end{array}\right.
\end{split}
\end{equation}
The case $x_1=x_2,n_2>n_1$ can be obtained by symmetry; for $x_1=x_2,n_1=n_2$ the determinant is zero.
\end{prop}

To prove Proposition~\ref{prop:hex:infinite} we first obtain a similar expression of a finite sub-graph $H=(V_H,E_H)$ of the honeycomb grid which admits a dimer covering, where however the sums needs to have a cut-off and a remainder. After taking the limit $H\to {\cal H}$ and removing the cut-off, one recovers Proposition~\ref{prop:hex:infinite}. Edge weights on $H$ are chosen to be identically $1$.

For the statement on $H$ we define the following subsets of vertices: for
$m\geq 1$
\begin{equation}
\check{\Sigma}_{m,x,n}=\bigcup_{i=0}^{m-1}
	\{ \bullet(x,n-i),\circ(x,n-i), \bullet(x+1,n-i-1),\circ(x+1,n-i) \}
\end{equation}
with $\check{\Sigma}_{0,x,n}=\emptyset$,
and for
$m\geq 0$
\begin{equation}
  \Sigma_{m,x,n}=\check{\Sigma}_{m,x,n}\cup\{\bullet(x,n-m),\circ(x+1,n-m)
  \}.
\end{equation}
Recall the notation $Z_G[V,E]$ from Definition \ref{def:ZEV}.
\begin{prop}\label{prop:hex:finite}
  Let $e_1,e_2$ be as in Proposition~\ref{prop:hex:infinite} and let
  $N_1,N_2$ be positive integers. Assume that the graph $H$ includes all
  vertices and edges appearing in the expressions below.  We have that
  $\det \left[(K_H)^{-1}(\circ(x_i+1,n_i),\bullet(x_j,n_j)) \right]_{1
    \leq i,j \leq 2}$ equals
\begin{equation}\label{eq:covarianceK1}
\sum_{m_1=1}^{N_1} \sum_{m_2=1}^{N_2} \Pb_H [ \tilde{O}_{m_1,e_1} \tilde{O}_{m_2,e_2}] +R^H_1
\end{equation}
if $|x_1-x_2|>1$ and
\begin{equation}\label{eq:covarianceK2}
\sum_{m_1=1}^{|n_1-n_2|-1} \sum_{m_2=1}^{N_2}\Pb_H [ \tilde{O}_{m_1,e_1} \tilde{O}_{m_2,e_2}] +\sum_{m_2=1}^{N_2}\Pb_H[\tilde{O}_{m_2+|n_1-n_2|,e_1}] + R_0^H
\end{equation}
if $x_1=x_2$ and $n_1>n_2$. The case $x_1=x_2$ and $n_2>n_1$ can be obtained by symmetry.

The remainder terms $R_i^H=R^H_i(x_1,n_1,N_1,x_2,n_2,N_2)$ for $i\in \{0,1\}$ are given by:
\begin{equation}
\begin{split}
&Z_H R^H_1 (x_1,n_1,N_1,x_2,n_2,N_2) = Z_H[\Sigma_{N_1,x_1,n_1} \cup \Sigma_{N_2,x_2,n_2}]\\&+
\sum_{m_2=1}^{N_2}Z_H [\Sigma_{N_1,x_1,n_1} \cup \check{\Sigma}_{m_2,x_2,n_2}] +
\sum_{m_1=1}^{N_1}Z_H[ \check{\Sigma}_{m_1,x_1,n_1} \cup {\Sigma}_{N_2,x_2,n_2}],
\end{split}
\end{equation}
and
\begin{multline}
  Z_H R^H_0(x,n_1,N_1,x,n_2,N_2) = \sum_{m_1=1}^{|n_1-n_2|-1} Z_H[\check{\Sigma}_{m_1,x,n_1}\cup \Sigma_{N_2,x,n_2}]\\
  +  Z_H[\Sigma_{|n_1-n_2|+N_2,x,n_1}].
\end{multline}
\end{prop}

\begin{proof}
  Consider the graph
  $H\backslash( \Sigma_{0,x_1,n_1} \cup \Sigma_{0,x_2,n_2})$. The
  vertices removed  from $H$ are on the same face (for each pair). This
  means that the Kasteleyn orientation of
  $H\backslash ( \Sigma_{0,x_1,n_1}\cup \Sigma_{0,x_2,n_2})$ is the
  same as that of $H$ (up to the removed vertices and their incident
  edges). An equivalent viewpoint is adding auxiliary edges between
  $\circ(x_i+1,n_i)$ and $\bullet(x_i,n_i)$ which must be covered by
  dimers for $i\in\{1,2\}$, and each auxiliary edge having orientation
  from $\circ(x_i+1,n_i)$ to $\bullet(x_i,n_i)$ for $i\in
  \{1,2\}$. This means that
  $Z_H[
  \Sigma_{0,x_1,n_1}]/{Z_H}=-(K_H)^{-1}(\circ(x_1+1,n_1),\bullet(x_1,n_1))$
  and
\begin{equation}\label{eq:hex:Kinv}
	\frac{ Z_H[ \Sigma_{0,x_1,n_1} \cup \Sigma_{0,x_2,n_2}]}{Z_H} =\det \left[ K_H^{-1}(\circ (x_i+1,n_i),\bullet (x_j,n_j)) \right]_{1\leq i,j \leq 2}.
\end{equation}
\begin{figure}
\begin{center}
\includegraphics[height=6cm]{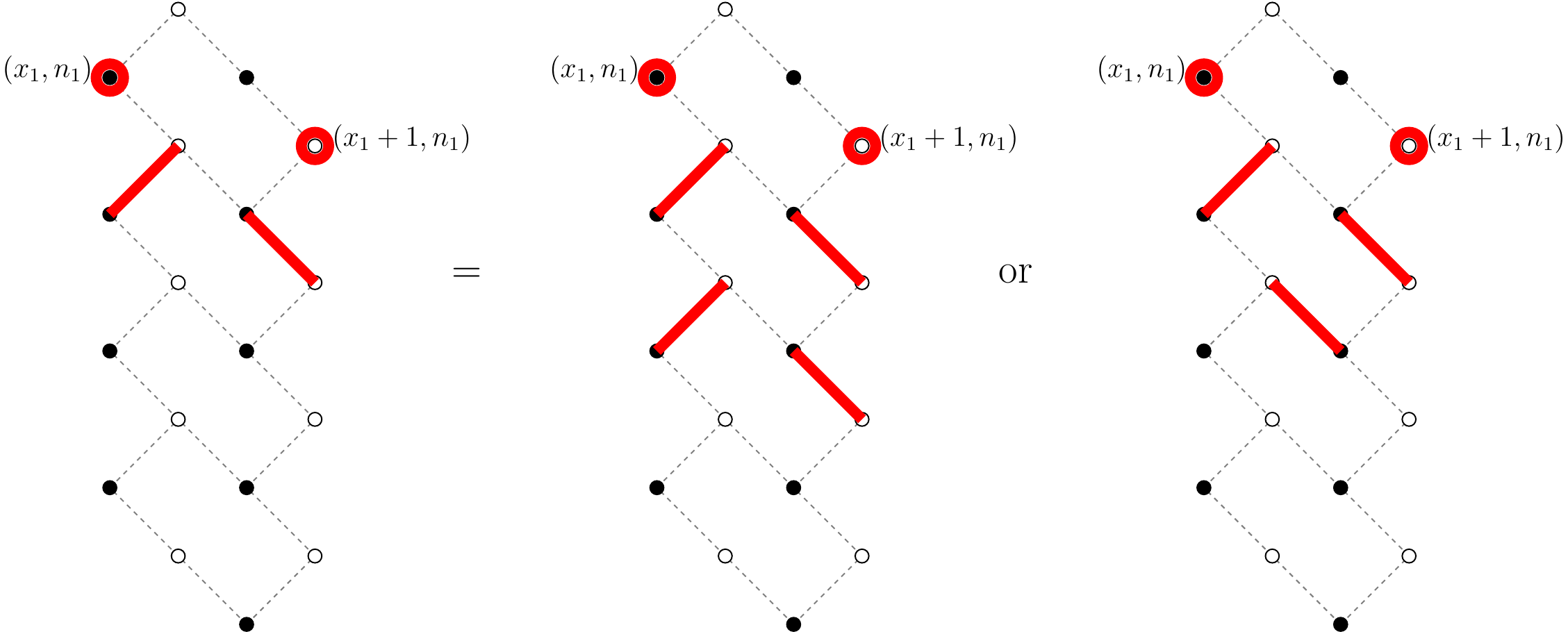}
\caption{ Vertices incident to red circles or red edges are those that are removed from the graph. The removed vertices on the left side are given by $\Sigma_{1,x_1,n_1}$. By considering the dimers incident to $\circ(x_1,n_1-1)$, this gives $\Sigma_{2,x_1,n_1}$ or $\check{\Sigma}_{2,x_1,n_1}$.  } \label{fig:hex1}
\end{center}
\end{figure}
\emph{Case 1: $|x_1-x_2|>1$.} We manipulate the dimer possibilities on $H\backslash (\Sigma_{0,x_1,n_1} \cup \Sigma_{0,x_2,n_2})$. Consider $H\backslash (\Sigma_{m_1,x_1,n_1} \cup \Sigma_{0,x_2,n_2})$ and the possible dimers incident to $\bullet (x_1,n-m_1)$; this leads to
\begin{equation}\label{eq:hex:iterate1}
  Z_H[ \Sigma_{m_1,x_1,n_1} \cup \Sigma_{0,x_2,n_2}] =
  Z_H[ \Sigma_{m_1+1,x_1,n_1} \cup \Sigma_{0,x_2,n_2}]+
  Z_H[ \check{\Sigma}_{m_1+1,x_1,n_1} \cup \Sigma_{0,x_2,n_2}],
\end{equation}
noting that $\Sigma_{m_1,x_1,n_1}\cup\{ \bullet(x_1+1,n_1-m_1-1),\circ(x_1,n_1-m_1)\}$ is the same as $\check{\Sigma}_{m_1+1,x_1,n_1}$; see Figure~\ref{fig:hex1}. Iterating~\eqref{eq:hex:iterate1} gives
\begin{equation}
          Z_H[ \Sigma_{0,x_1,n_1} \cup \Sigma_{0,x_2,n_2}]=
          Z_H[ \Sigma_{N_1,x_1,n_1} \cup \Sigma_{0,x_2,n_2}]+
          \sum_{m_1=1}^{N_1} Z_H[\check{\Sigma}_{m_1,x_1,n_1} \cup \Sigma_{0,x_2,n_2}].
\end{equation}
Since the set $\Sigma_{m_2,x_2,n_2}$ for $0\leq m_2 \leq N_2$ does
not intersect $\Sigma_{m_1,x_1,n_1}$ for $0\leq m_1 \leq N_1$,
applying an analogous procedure given above to
$Z_H[ \Sigma_{N_1,x_1,n_1} \cup \Sigma_{0,x_2,n_2}]$ and
$Z_H[ \check{\Sigma}_{m_1,x_1,n_1} \cup \Sigma_{0,x_2,n_2}]$ gives
\begin{equation}
Z_H[ \Sigma_{0,x_1,n_1} \cup \Sigma_{0,x_2,n_2}]=R^H_1Z_H+\sum_{m_1=1}^{N_1} \sum_{m_2=1}^{N_2}
Z_H[ \check{\Sigma}_{m_1,x_1,n_1} \cup \check{\Sigma}_{m_2,x_2,n_2}],
\end{equation}
where $R^H_1$ is as given in the statement of the proposition. Dividing both sides of the above equation by $Z_H$, noting~\eqref{eq:hex:Kinv} and
\begin{equation}
  \label{en2}
\frac{Z_H[ \check{\Sigma}_{m_1,x_1,n_1}\cup \check{\Sigma}_{m_2,x_2,n_2}] }{Z_H} =
\Pb_H[\tilde{O}_{m_1,e_1} \tilde{O}_{m_2,e_2}]
\end{equation}
gives~\eqref{eq:covarianceK1}.

\begin{figure}
		\begin{center}
		\includegraphics[height=6cm]{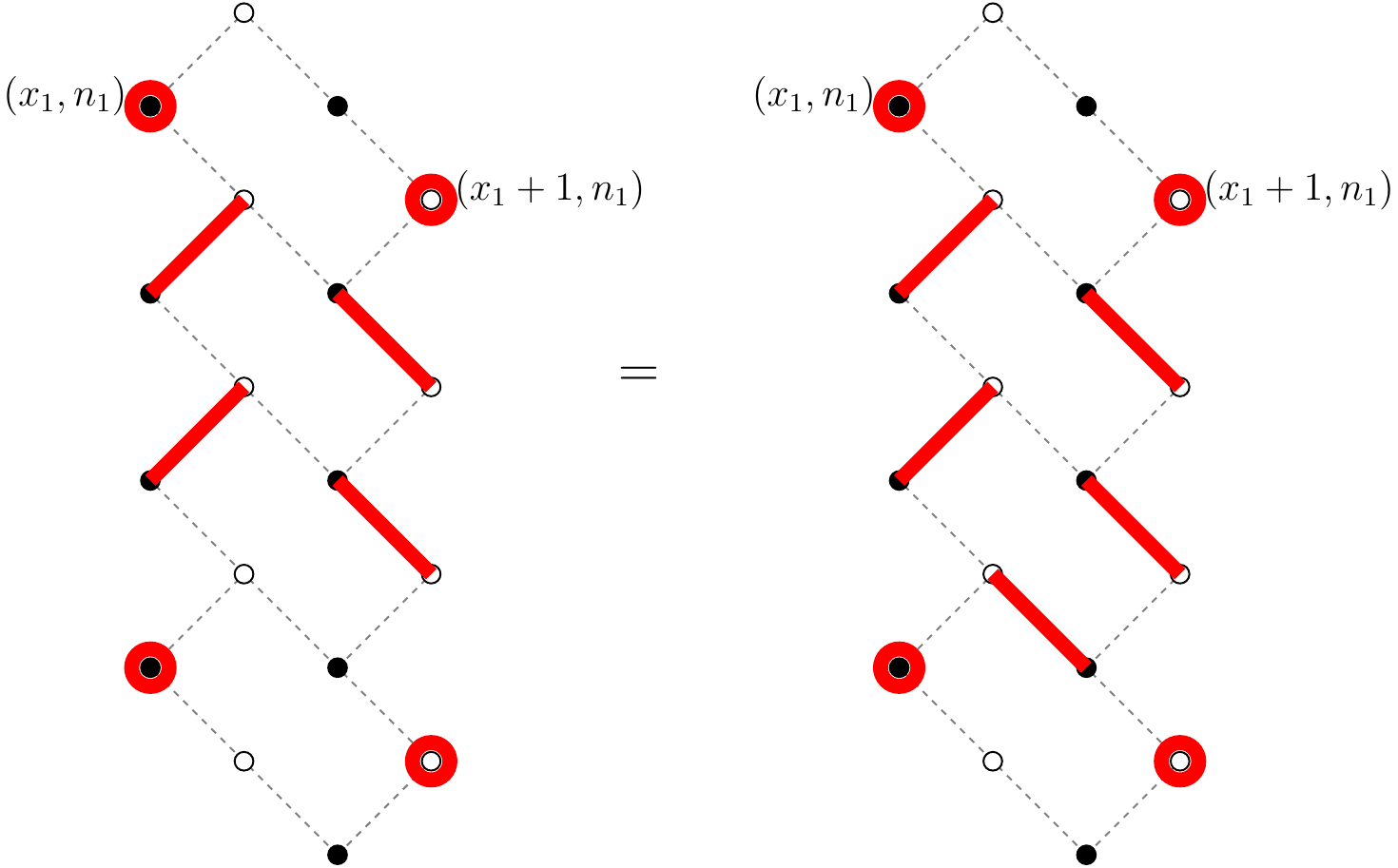}
			\caption{Using the same conventions as in Figure~\ref{fig:hex1}, the left side shows $\Sigma_{2,x_1,n_1} \cup \Sigma_{0,x_2,n_2}$ removed from the graph with $x_1-x_2=0$ and $n_1=n_2+3$. The right side shows the same configuration but with the edge that is forced to be covered by a dimer. }
		\end{center}
	\end{figure}
        \emph{Case 2: $x_1=x_2\equiv x$.} We give the computation for
        $n_1>n_2$ because the computation is similar for $n_1<n_2$. By
        noting that there is a single choice of dimer incident to the
        vertex $\circ(x,n_2+1)$ on the graph
        $H\backslash (\Sigma_{|n_1-n_2|-1,x,n_1} \cup
        \Sigma_{m_2,x,n_2})$ which is given by the edge
	$(\circ(x,n_2+1),\bullet(x+1,n_2))$, that $\Sigma_{|n_1-n_2|-1,x,n_1} \cup \{ \circ(x,n_2+1),\bullet(x+1,n_2)\}=\check{\Sigma}_{|n_1-n_2|,x,n_1}$,
	and that
        $\Sigma_{|n_1-n_2|,x,n_1} \cup
        \Sigma_{m_2,x,n_2}=\Sigma_{|n_1-n_2|+m_2,x,n_1}$, then by
        following the steps given for $|x_1-x_2|>1$, we have
	\begin{equation}
		\begin{split}
                  Z_H[\Sigma_{0,x,n_1}\cup \Sigma_{0,x,n_2}] &= \sum_{m_1=1}^{|n_1-n_2|-1}
                  Z_H[ \check{\Sigma}_{m_1,x,n_1} \cup\Sigma_{N_2,x,n_2}]
		+	\sum_{m_2=1}^{N_2}Z_H[\check{\Sigma}_{|n_1-n_2|+m_2,x,n_1}] \\
			&+Z_H[\Sigma_{|n_1-n_2|+N_2,x,n_1}]+
		\sum_{m_1=1}^{|n_1-n_2|-1} \sum_{m_2=1}^{N_2}
                  Z_H[\check{\Sigma}_{m_1,x,n_1} \cup \check{\Sigma}_{m_2,x,n_2}].                                     \end{split}
              \end{equation}
Dividing both sides of the above equation by $Z_H$, noting~\eqref{eq:hex:Kinv} and
\begin{equation}
	\frac{Z_H[\check{\Sigma}_{|n_1-n_2|+m_2,x,n_1}]}{Z_H} = \Pb_H[ \tilde{O}_{m_2+|n_1-n_2|,e_1}]
\end{equation}
leads to~\eqref{eq:covarianceK2}.
\end{proof}

\begin{proof}[Proof of Proposition~\ref{prop:hex:infinite}]
  The first step is to provide bounds for $R^H_0$ and $R^H_1$ in terms
  of probabilities. This is achieved by using the same argument given
  in Lemma~\ref{lem:RlGsquare} to each of the terms found in these
  expressions. That is, we have that
  \begin{equation}
    \label{en}
  Z_H [\Sigma_{N_1,x_1,n_1}\cup V] \leq Z_H [ \check{\Sigma}_{N_1, x_1,n_1}\cup V],
\end{equation}
 where $V$ denotes a set of removed vertices
which are not incident to the edges incident to $\Sigma_{N_1,x_1,n_1}$.  To
verify this equation, recall that
$\Sigma_{N_1,x_1,n_1}= \check{\Sigma}_{N_1, x_1,n_1} \cup \{
\bullet(x_1,n_1-N_1),\circ(x_1+1,n_1-N_1)\}$. Since the two
additional vertices are on the same face, the Kasteleyn orientation on
the graphs $H\backslash( \Sigma_{N_1,x_1,n_1}\cup V)$ and
$H\backslash (\check{\Sigma}_{N_1, x_1,n_1}\cup V)$ are the same up to these
two additional vertices, which means that
	\begin{equation}
		\frac{Z_H [\Sigma_{N_1,x_1,n_1}\cup V] }{Z_H [ \check{\Sigma}_{N_1, x_1,n_1}\cup V]} = \frac{ \det (K_{H\backslash( \check{\Sigma}_{N_1, x_1,n_1}\cup V)}|_{ \bullet(x_1,n_1-N_1),\circ(x_1+1,n_1-N_1)} )}{ \det (K_{H\backslash (\check{\Sigma}_{N_1, x_1,n_1}\cup V)})} \leq 1
	\end{equation}
	because each term in the expansion of the determinant in the
        numerator is also present in the denominator.
        We show how to use the above inequalities to bound each of the
        terms in $R^H_0$ and $R^H_1$ by using the term
        $Z_H [\Sigma_{N_1,x_1,n_1}\cup \Sigma_{N_2,x_2,n_2}]/Z_H$ as
        an example. The rest of the terms follow by similar
        computations.  From \eqref{en} and \eqref{en2} we obtain
	\begin{equation}\label{eq:hex:bound1}	
          \frac{Z_H [\Sigma_{N_1,x_1,n_1}\cup \Sigma_{N_2,x_2,n_2}]}{Z_H} \leq
          \Pb_H[\tilde{O}_{N_1,e_1}\tilde{O}_{N_2,e_2}].
\end{equation}
Using this and analogous bounds we can estimate the error terms
$R^H_0,R^H_1$ as
\begin{equation} \label{eq:hex:boundrem1}
0\leq 	{R^H_1} \leq \Pb_H[\tilde{O}_{N_1,e_1} \tilde{O}_{N_2,e_2}]+ \sum_{m_2=1}^{N_2} \Pb_H[ \tilde{O}_{N_1,e_1}\tilde{O}_{m_2,e_2}]+ \sum_{m_1=1}^{N_1} \Pb_H[ \tilde{O}_{m_1,e_1}\tilde{O}_{N_2,e_2}]
\end{equation}
and
\begin{equation} \label{Rzero}
0\leq {R^H_0} \leq \sum_{m_1=1}^{|n_1-n_2|-1} \Pb_H[\tilde{O}_{m_1,e_1} \tilde{O}_{N_2,e_2}]+\Pb_H[\tilde{O}_{|n_1-n_2|+N_2,e_1}].
\end{equation}

\begin{figure}
\begin{center}
\includegraphics[height=6cm]{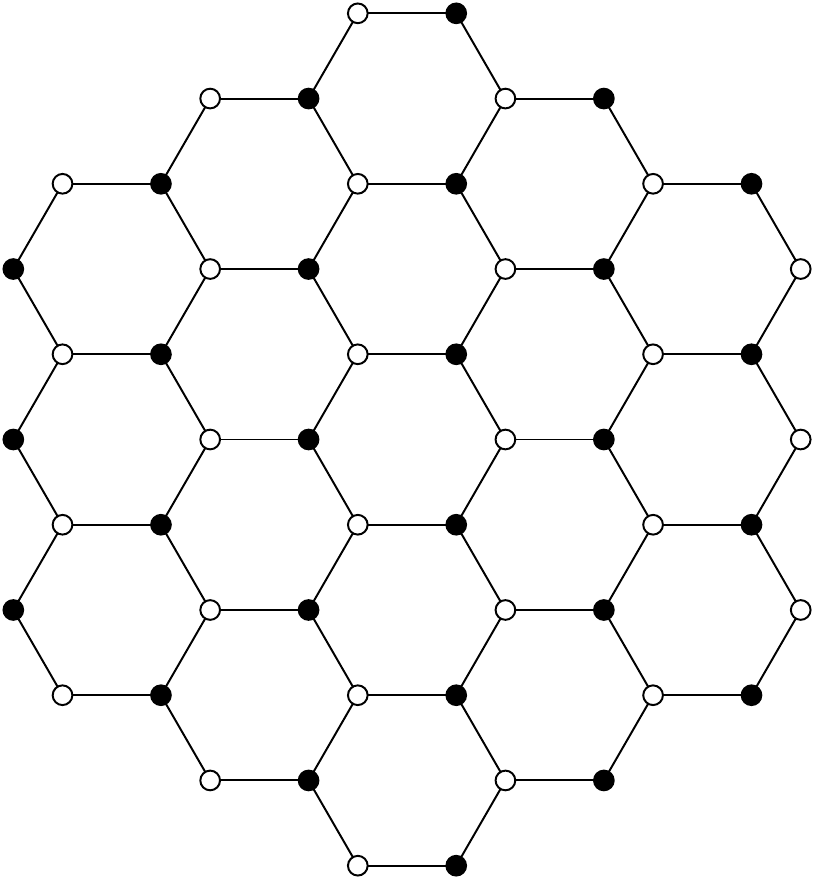}
	\caption{The hexagonal graph $H_L$ for $L=3$.}
\label{fig:honeycomb}
\end{center}
\end{figure}

For the moment we did not specify the finite graph $H$: now we take it to be
a $L\times L\times L$ hexagonal subset $H=H_L$ of $\mathcal H$ (see Figure \ref{fig:honeycomb} for $L=3$). Non-zero entries $K_{H_L}(b,w)$ of its Kasteleyn matrix are chosen to be all equal to $1$.  We will let $L$ grow to infinity
and, in a second stage, we will let $N_1,N_2$ in Proposition \ref{prop:hex:finite} tend to infinity.

Similar to uniform random tilings of Aztec diamonds, uniform
dimer coverings of a large hexagon exhibit a limit shape
phenomenon~\cite{CLP98}, that we briefly recall. Rescale $H_L$ by a
factor $1/L$ so that it converges to a hexagon $H_\infty$ of side
length $1$ as $L\to\infty$ and let $\mathcal D\subset H_\infty$ be the
open disk tangent to the six sides of $H_\infty$.  Let
$\xi \in \mathcal D$ and let $\hat H_L$ be the graph
$H_L$ translated by $-\lfloor \xi L\rfloor$. Then, Theorem 2 in
\cite{Pet14} says that the local statistics under $\Pb_{\hat H_L}$
converges to that of a Gibbs measure $\pi^{\mathcal H}_{\hat \rho}$, for a
certain (rather explicit) $\hat\rho=\hat\rho(\xi_1,\xi_2)\in P^{\mathcal H}$. Moreover,
Proposition~7.10 in \cite{Pet14} implies that the inverse Kasteleyn
matrix $K^{-1}_{\hat H_L}$ satisfies, for any fixed vertices $\circ(x_1,x_2),\bullet (y_1,y_2)$,
\begin{multline}
  \label{eq:Petrov}
  \lim_{L\to\infty}K^{-1}_{\hat H_L}(\circ(x_1,x_2),\bullet (y_1,y_2))\\=a_2 \left(\frac{a_1}{a_3}
  \right)^{y_1-x_1} \left( \frac{a_2}{a_3}\right)^{y_2-x_2}
  \overline{K}^{-1}(\circ(x_1,x_2),\bullet (y_1,y_2)).
\end{multline}
This is the analog of \eqref{eq:gdn} for the square lattice.
Here, $\overline{K}^{-1}$ is the infinite inverse Kasteleyn matrix \eqref{hexagonK-1}, with $a_i=a_i(\hat \rho)$ (recall that edge weights are a function of the slope, as discussed just before \eqref{hexagonK-1}).
As in \eqref{eq:gdn}, the exponential pre-factors in the r.h.s. arise from the gauge transformation relating the
Kasteleyn matrix $K_{\hat H_L}$, with weights $1$, to that of the infinite lattice $\mathcal H$, with weights $a_1,a_2,a_3$.
Moreover, for every $\rho\in P^{\mathcal H}$ it is possible to find $\xi\in\mathcal D$ such that $\hat\rho(\xi)=\rho$, and this is how we fix $\xi$.

Note that Petrov's results hold for more general regions than
hexagonal ones, but we do not need this level of generality
here.

We let $L\to\infty$ now. From \eqref{eq:Petrov} we see that
\begin{multline}
  \lim_{L\to \infty}
 \det \left( {K}^{-1}_{\hat H_L}(\circ(x_i+1,n_i),\bullet(x_j,n_j)) \right)_{1 \leq i,j \leq 2} \\=\frac{(a_2a_3)^2}{a_1^2}\det \left( \overline{K}^{-1}(\circ(x_i+1,n_i),\bullet(x_j,n_j)) \right)_{1 \leq i,j \leq 2}.
\end{multline}
Also, from the above discussion we see that all the $\Pb_{\hat H_L}$ probabilities in \eqref{eq:covarianceK1} and
\eqref{eq:covarianceK2} tend to the corresponding $\pi^{\mathcal H}_\rho$ probabilities, while (recall \eqref{Rzero})
\begin{equation}
0\leq \liminf_{L\to \infty} R^{\hat H_L}_0 \leq  \limsup_{L\to \infty} R^{\hat H_L}_0 \leq \sum_{m_1=1}^{|n_1-n_2|-1} \pi_{\rho}^{\mathcal{H}}[\tilde{O}_{m_1,e_1} \tilde{O}_{N_2,e_2}]+\pi^{\mathcal H}_\rho(\tilde{O}_{|n_1-n_2|+N_2,e_1}),
\end{equation}
and a similar bound for $ R^{\hat H_L}_1$ from~\eqref{eq:hex:boundrem1}.

Thus we have shown that
$\frac{(a_2a_3)^2}{a_1^2} \det \left(
  (K^{-1}(\circ(x_i+1,n_i),\bullet(x_j,n_j)) \right)_{1 \leq i,j \leq
  2}$, which is independent of $N_1$ and $N_2$, it is a sum of
non-negative terms. This implies that the sums are convergent and
consequently the remainder terms tends to zero as $N_1$ and $N_2$ tend
to infinity (this can be deduced also from Eq. \eqref{eq:disas4}
below). The statement of Proposition \ref{prop:hex:infinite} follows.
\end{proof}

\begin{rem}
  The above proof uses a hexagonal finite graph $H_L$ and the convergence
  of its inverse Kasteleyn matrix to that of $\mathcal{H}$. The
  approach to compute the speed of growth used in the published
  version of~\cite{CF15}(Proposition 3.7) uses instead a
  toroidal graph $T_L$ with $L\to\infty$ but it contains mistakes
  since it does not take into account the fact that on the torus the
  terms in the expansion of the determinant of $K_{T_L}$ come with
  different signs. Although that argument could be adjusted, it is
  much simpler to use the planar hexagonal graph $H_L$ instead, as we do
  here.
\end{rem}

\subsection{Proof of
  Proposition~\ref{prop:evenoddest}} \label{subsec:estimates} The two
variance terms on the right side of (\ref{eq4.2}) are bounded in the
same way. We therefore presents the details for only one.  Recall that
$\Lambda_L^0$ consists of all horizontal edges
$e=(\bullet(x+1,n),\circ(x,n+1))$ with $x,n \in [0,L]$ and even $x$.
For several bounds we will use (this can be deduced from Lemma A.1
of~\cite{Ton15})
\begin{equation}\label{eq:disas4}
\pi_\rho^{\mathcal{H}} (O_{m,e})=\pi_\rho^{\mathcal{H}} (\widetilde O_{m,e})\le C_1 e^{-c_1 m}
\end{equation}
where $C_1,c>0$ are constants depending on $\rho$.

Recall that $\widehat V(e)=\sum_{m\geq 1} \Id_{O_{m,e}}$ and $\widetilde V(e)=\sum_{m\geq 1} \Id_{\widetilde O_{m,e}}$. Also, we define
\begin{equation}
  D(e)=\sum_{\substack{m\geq 1:\\ \widetilde O_{m,e}\not\subset\Lambda_L^0}}\Id_{\widetilde O_{m,e}},\quad U(e)=\sum_{\substack{m\geq 1:\\ O_{m,e}\not\subset\Lambda_L^0}}\Id_{O_{m,e}}.
\end{equation}
Recall that $O_{m_i,e_i}=\widetilde O_{m_i,e_i+m_i}$ and notice the
following bijection: for $(m,e)$ such that $O_{m,e}\subset \Lambda_L^0$,
there exists a unique pair $(\tilde m,\tilde e)$ such that
$O_{m,e}=\widetilde O_{\tilde m,\tilde e}$, namely $\tilde m=m$ and
$\tilde e=e+m$ (and vice versa). This gives
\begin{equation}
\sum_{e\in\Lambda_L^0}\widehat V(e)=\sum_{e\in\Lambda_L^0} \widetilde V(e)+\sum_{e\in\Lambda_L^0} U(e)-\sum_{e\in\Lambda_L^0} D(e).
\end{equation}
By Cauchy-Schwarz, it is enough to bound the variances of the three sums above. Let us first bound the variance of the sum of $U(e)$. Bounding the variance by the second moment we get
\begin{equation}\label{eq4.40}
  \mathrm{Var}_{\pi_{\rho}^{\mathcal{H}}}\bigg(\sum_{e \in \Lambda_L^0}U(e)\bigg)
  \leq \sum_{e_1,e_2\in \Lambda_L^0} \sum_{\substack{m_1,m_2\geq 1:\\ O_{m_1,e_1}\not\subset\Lambda_L^0\\O_{m_2,e_2}\not\subset\Lambda_L^0}}{\pi_{\rho}^{\mathcal{H}}}(O_{m_1,e_1} O_{m_2,e_2}).
\end{equation}
Using (\ref{eq:disas4}), $\pi_{\rho}^{\mathcal{H}}(O_{m_1,e_1} O_{m_2,e_2})\leq \min\{\pi_{\rho}^{\mathcal{H}}(O_{m_1,e_1}),\pi_{\rho}^{\mathcal{H}}(O_{m_2,e_2})\}$ and $\min\{e^{-c x},e^{-c y}\}\leq e^{-c (x+y)/2}$ for $x,y\geq 0$, we get
\begin{equation}
(\ref{eq4.40})\leq \bigg(C_1 \sum_{e\in\Lambda_L^0}\sum_{\substack{m\geq 1:\\ O_{m,e}\not\in\Lambda_L^0}} e^{-c_1 m/2}\bigg)^2= \Or(L^2).
\end{equation}
Similarly one gets the bound for the variance of $\sum_{e\in\Lambda_L^0} D(e)$.

For the main term, we have
\begin{equation}\label{eq4.38}
\mathrm{Var}_{\pi_{\rho}^{\mathcal{H}}}\bigg(\sum_{e \in \Lambda_L^0}\widetilde V(e)\bigg) =
\sum_{e_1,e_2 \in \Lambda_L^0}\left({\pi_{\rho}^{\mathcal{H}}}[\widetilde V(e_1) \widetilde V(e_2)]-{\pi_{\rho}^{\mathcal{H}}}[\widetilde V(e_1)]{\pi_{\rho}^{\mathcal{H}}}[\widetilde V(e_2)]\right).
\end{equation}
Using (\ref{eqSpeedKernel}) we have
\begin{equation}\label{eq4.34}
\sum_{e_1,e_2 \in \Lambda_L^0}{\pi_{\rho}^{\mathcal{H}}}[\widetilde V(e_1)]{\pi_{\rho}^{\mathcal{H}}}[\widetilde V(e_2)] =\sum_{e_1,e_2 \in \Lambda_L ^0}\left(\frac{a_2a_3}{a_1} \right)^2 \prod_{i=1}^2\overline{K}^{-1} (\circ(x_i+1,n_i),\bullet (x_{i},n_{i})).
\end{equation}
Propositions~\ref{prop:hex:variance} and~\ref{prop:hex:infinite} imply that
\begin{multline}\label{eq4.35}
\sum_{e_1,e_2 \in \Lambda_L^0}{\pi_{\rho}^{\mathcal{H}}}[\widetilde V(e_1) \widetilde V(e_2)] = E_\Lambda\\+
\sum_{e_1,e_2 \in \Lambda_L^0} \left(\frac{a_2a_3}{a_1} \right)^2 \det\left[ \overline{K}^{-1} (\circ(x_i+1,n_i),\bullet (x_j,n_j)) \right]_{1 \leq i,j\leq 2}
\end{multline}
where the error term $E_\Lambda$ is given by
\begin{equation}
E_\Lambda = \sum_{\substack{e_1,e_2 \in \Lambda_L^0;\\ x_1=x_2}}\sum_{m\geq 1}(2m-1){\pi_{\rho}^{\mathcal{H}}}[\widetilde O_{m+|n_2-n_1|,e_1\vee e_2}]
\end{equation}
and $e_1\vee e_2=e_1\Id_{n_1>n_2}+e_2\Id_{n_1<n_2}$. We recall that
$e_j=(\bullet(x_j+1,n_j),\circ(x_j,n_j+1))$ as in Proposition \ref{prop:hex:variance}.
Plugging (\ref{eq4.34}) and (\ref{eq4.35}) into (\ref{eq4.38}) leads to
\begin{equation}\label{eq4.36}
\mathrm{Var}_{\pi_{\rho}^{\mathcal{H}}}\bigg(\sum_{e \in \Lambda_L^0}\widetilde V(e)\bigg) = E_\Lambda + \sum_{e_1,e_2 \in \Lambda_L^0}\left(\frac{a_2a_3}{a_1} \right)^2 \prod_{i=1}^2\overline{K}^{-1} (\circ(x_i+1,n_i),\bullet (x_{i+1},n_{i+1})),
\end{equation}
with $x_3\equiv x_1$ and $n_3\equiv n_1$. Thus it remains to bound the
two terms in (\ref{eq4.36}).

The leading term is bounded as follows. We have
\begin{equation} \label{eq:hex:integralbound}
\left| \overline{K}^{-1}(\circ(x_i+1,n_i),\bullet(x_{i+1},n_{i+1})) \right|
\leq \frac{C}{1+|n_1-n_2|+|x_1-x_2|},
\end{equation}
where $C=C(\rho)$. The above bound follows from the computations given in the proof of Lemma~4.4 in~\cite{KOS03}. We omit details. Thus we have
\begin{equation}
\begin{aligned}
&\sum_{e_1,e_2 \in \Lambda_L^0} \bigg|\prod_{i=1}^2 \overline{K}^{-1}(\circ(x_i+1,n_i),\bullet(x_{i+1},n_{i+1})) \bigg| \\
\leq &\sum_{x_1,n_1,x_2,n_2\in[0,L]}\frac{C'}{1+|n_1-n_2|^2+|x_1-x_2|^2} \leq C'' L^2\log{L}
\end{aligned}
\end{equation}
as wished.
Finally, using (\ref{eq:disas4}), $E_\Lambda$ is bounded by
\begin{equation}
  C_1
\sum_{\substack{e_1,e_2 \in \Lambda_L^0;\\ x_1=x_2}}\sum_{m\geq 1} (2m-1)e^{-c_1 (m+|n_1-n_2|)}=\Or(L^2).
\end{equation}
This completes the proof of Proposition~\ref{prop:evenoddest}.

\appendix

\section{Formulas on $\mathbb{Z}^2$}\label{App:Positivity}

In this section, we give formulas useful for the inverse Kasteleyn matrix on $\mathbb{Z}^2$.

\begin{lem}\label{appendixlemma}
	For $\mathcal{G}=\mathbb{Z}^2$ and provided that $|\tanh B_2 \cosh B_1 |\leq 1$ then
	\begin{equation} \label{eq:KfullZ2}
\overline{K}^{-1} (\circ (x_1,x_2), \bullet (y_1,y_2)) =\frac{f(x_1,x_2,y_1,y_2)}{2\pi \mathrm{i}} \int_C \frac{z^{y_1-x_1-1}(z-1)^{y_2-x_2}}{(z+1)^{y_2-x_2+1}} d z
\end{equation}
	where $f(x_1,x_2,y_1,y_2)= \mathrm{i}^{y_1-x_1+x_2-y_2-1}e^{B_2(y_2-x_2)}e^{B_1(y_1-x_1-1)}$ and $C$ is a contour from $\overline{\Omega}_c$ to $\Omega_c=e^{-B_1}e^{\mathrm{i} \arccos( \cosh(B_1) \tanh(B_2))}$ passing to the right of the origin if $y_2 \geq x_2$ and passing to the left of the origin if $y_2<x_2$.
\end{lem}

Note that the condition $ |\tanh B_2 \cosh B_1 |\leq 1$ is the one that defines the amoeba $\mathcal B$ in \eqref{eq:ameba}.

\begin{proof}[Proof of Lemma~\ref{appendixlemma}]
 In~\eqref{eq:wholeplaneK}, we make the change of
 variables $z_1 \mapsto z_1e^{B_1}$ and $z_2 \mapsto z_2e^{B_2}$
 which gives
	\begin{equation}
\overline{K}^{-1} (\circ (x_1,x_2), \bullet (y_1,y_2)) =
		\frac{ e^{B_1(y_1-x_1-1)+B_2(y_2-x_2)}}{(2\pi \mathrm{i})^2}\int \frac{d z_1}{z_1}\frac{d z_2}{z_2}
\frac{z_1^{y_1-x_1-1}z_2^{y_2-x_2}}{
 1+ \mathrm{i}z_1^{-1} +\mathrm{i}z_2^{-1}+z_1^{-1} z_2^{-1}}
\end{equation}
for $w=\circ(x_1,x_2),b=\bullet(y_1,y_2)$ where the integral are over
positive contours $|z_1|=e^{-B_1}$ and $|z_2|=e^{-B_2}$.

	We make the change of variables $\omega= (\mathrm{i}-z_2)/(\mathrm{i}+z_2)$ (i.e., $z_2=-\mathrm{i}(\omega-1)/(1+\omega)$) and $z_1=z\mathrm{i}$
	which gives
\begin{equation}
	\frac{f(x_1,x_2,y_1,y_2)}{(2 \pi \mathrm{i})^2} \int {dz} \int d\omega \frac{(\omega-1)^{y_2-x_2}}{(\omega+1)^{y_2-x_2+1}}\frac{1}{\omega-z},
\end{equation}
where the contour for $|z|=e^{-B_1}$ is positively oriented and the contour for $\omega$ is explained below. Taking the residue at $\omega=z$ gives the integral described in the lemma. It remains to find the contours under these transformations, ascertain that there are no other contributions, and finally verify the intersection points.

The map $\omega=(\mathrm{i}-z_2)/(\mathrm{i}+z_2)$ maps the positively oriented contour $|z_2|=e^{-B_2}$ to:
\begin{itemize}
\item[(a)] for $B_2>0$, to a positively oriented circle having positive real part, center on the real axis, and including $1$,
\item[(b)] for $B_2=0$, to the imaginary line from $\infty\mathrm{i}$ to $-\infty\mathrm{i}$,
\item[(c)] for $B_2<0$, to a negatively oriented circle having negative real part, center on the real axis, and including $-1$.
\end{itemize}

In particular, these contours intersect with $|z|=e^{-B_1}$ if and only if $|\tanh B_1 \cosh B_2 |\leq 1$. From now we consider this restriction of the values of $B_1,B_2$. The two intersections are complex conjugate complex numbers $\Omega_c$ and $\overline \Omega_c$, with the convention that $\Im(\Omega_c)\geq 0$. A simple geometric computation gives $\Omega_c=e^{-B_1} e^{\mathrm{i} \arccos (\cosh( B_1) \tanh (B_2))}$.

Notice that the possible poles in $\omega$ are $\omega=z$ and
$\omega=\pm 1$. The residue at infinity is zero and therefore, by Cauchy
residue's theorem, for any value of $B_2$ we can choose to perform the
integral over $\omega$:
\begin{itemize}
\item[(A)] either along a positively oriented path enclosing the poles at $1$ and at the portion of $z$ from $\overline\Omega_c$ to $\Omega_c$,
\item[(B)] or along a negatively oriented path enclosing the poles at
 $-1$ and at the portion of $z$ from $\Omega_c$ to
 $\overline\Omega_c$.
\end{itemize}

The idea is now to choose between option (A) and (B) for the contours
in such a way that the poles at $\pm 1$ are never inside the contour
for $\omega$, so that we are left (at most) with the pole at $z$
only.\smallskip

\emph{Case 1: $y_2-x_2\geq 0$}. In this case there is a pole at $-1$ and we choose the contour for $\omega$ as in (A). The residue at $\omega=z$ gives the claimed result.\smallskip

\emph{Case 2: $y_2-x_2<0$}. In this case there is a pole at $1$ and we choose the contour for $\omega$ as in (B). The residue at $\omega=z$ gives, due to the orientation of the contour, $-1$ times the integral over $z$ from $\Omega_c$ to $\overline\Omega_c$, which can be equivalently be though to be the integral over $z$ from $\overline\Omega_c$ to $\Omega_c$ passing to the left of the origin.
\end{proof}

	\begin{lem} \label{lem:Z2slopes}
For $\mathcal{G}=\mathbb{Z}^2$ and provided that $ |\tanh (B_2) \cosh (B_1) |\leq 1$
\begin{equation}
\rho_1=-\frac{1}{2} +\frac{1}{\pi}\arg \Omega_c
\end{equation}
and
\begin{equation}
\rho_2=\frac{\arg (\Omega_c -1)}{\pi }-\frac{\arg (\Omega_c +1)}{\pi }-\frac{1}{2}
\end{equation}
	where $\Omega_c$ is defined in Lemma~\ref{appendixlemma}.
\end{lem}
\begin{proof}
Using the integral formula for $\overline{K}^{-1}$ found in Lemma~\ref{appendixlemma}, we have from~\eqref{eq:Zslope1}
\begin{equation}\begin{split}
\rho_1 		&= \frac{1}{2} - \frac{1}{2\pi } \int \left(\frac{\mathrm{i} }{\omega-1} -\frac{\mathrm{i} }{(\omega-1)\omega}\right) d\omega \\
&=\frac{1}{2}-\frac{1}{2\pi } \int \frac{\mathrm{i}}{\omega}d\omega,
\end{split}
\end{equation}
where the integral goes between $\overline{\Omega}_c$ and $\Omega_c$ and
passes to the left of the origin. The formula for $\rho_1$ follows
from evaluating the above integral. The formula for $\rho_2$ follows
from a similar computation.
		
\end{proof}

\begin{lem}\label{lem:Z2entries}
 For $\mathcal{G}=\mathbb{Z}^2$ provided that
 $ |\tanh( B_2) \cosh( B_1)| \leq 1$ then
\begin{equation}
\overline{K}^{-1}(\circ(-1,1),\bullet(1,0))=\frac{e^{B_1-B_2} \left(-\arg \left(\Omega _c-1\right)-\Im\left(\Omega _c\right)+\pi \right)}{\pi },
\end{equation}
and
\begin{equation}
\overline{K}^{-1}(\circ(0,1),\bullet(1,0))=-\frac{\mathrm{i} e^{-B_2} \left(\pi -\arg \left(\Omega _c-1\right)\right)}{\pi }.
\end{equation}
\end{lem}
\begin{proof}
These follow from evaluating the appropriate single integral formulas from Lemma~\ref{appendixlemma}.
\end{proof}

\begin{lem}\label{lem:positivity}
We have that
\begin{equation}
		\begin{split}
			&\frac1 {e^{B_1-B_2}}\left( \bar{K}^{-1}({\w},{\b}) + 2\frac{ \bar K({\b_1} ,\tilde{\w})}{ \bar K({\b_1}, {\w}) \bar K({\b} ,\tilde{\w})} \pi^{\mathbb Z^2}_\rho[\tilde{e}_0,e_1] \right) +\frac{\Im \Omega_c}{\pi}\\
			&=- \pi^{\mathbb Z^2}_\rho[\tilde{e}_0,e_1]+ \pi^{\mathbb Z^2}_\rho[\tilde{e}_0^c,e_1],
		\end{split}
\end{equation}
	\begin{equation} \label{eq:positivity}
		- \pi^{\mathbb Z^2}_\rho[\tilde{e}_0,e_1]+ \pi^{\mathbb Z^2}_\rho[\tilde{e}_0^c,e_1] -\frac{\Im \Omega_c}{\pi}<0,
	\end{equation}
and
	\begin{equation}
		\Im \Omega_c=\frac{ \sin ^2\psi _1}{ \tan \psi _2} +\frac{\sin \psi _1}{ \sin \psi _2 } \sqrt{ \sin ^2\psi _2+\sin ^2\psi _1 \cos ^2\psi _2}
	\end{equation}
	where $\Omega_c$ is defined in Lemma~\ref{appendixlemma}, $\psi_i=(1/2+\rho_i)\pi$ for $i\in\{1,2\}$.

\end{lem}
	
\begin{proof}	
	By noticing that
	\begin{equation}
\frac2 {e^{B_1-B_2}}\frac{ \bar K({\b_1} ,\tilde{\w})}{ \bar K({\b_1}, {\w}) \bar K({\b} ,\tilde{\w})} =
		\frac2 {e^{B_1-B_2}} \frac{e^{B_1+B_2}}{\mathrm{i}^2 e^{2 B_2}}= -2
	\end{equation}
and	
	\begin{equation}
- \pi^{\mathbb Z^2}_\rho[\tilde{e}_0,e_1]+ \pi^{\mathbb Z^2}_\rho[\tilde{e}_0^c,e_1]=-2 \pi^{\mathbb Z^2}_\rho[\tilde{e}_0,e_1]+ \pi^{\mathbb Z^2}_\rho[e_1],
	\end{equation}
	the first equation follows by comparing $\frac1 {e^{B_1-B_2}} \bar{K}^{-1}(\circ(-1,1),\bullet(1,0))$ and $ \pi^{\mathbb Z^2}_\rho[e_1]=\overline{K}(\bullet(1,0),\circ(0,1)) \overline{K}^{-1}(\circ(0,1),\bullet(1,0))=(\mathrm{i} e^{B_2})\overline{K}^{-1}(\circ(0,1),\bullet(1,0))$ which are both given in Lemma~\ref{lem:Z2entries}.

To verify~\eqref{eq:positivity}, we have using Lemma~\ref{lem:Z2entries} and \eqref{eq:localstatsG} to compute $ \pi^{\mathbb Z^2}_\rho[\tilde{e}_0,e_1]$
	\begin{equation}
 \begin{split}
 \label{eq:bracket}
			&\frac1 {e^{B_1-B_2}}\left( \bar{K}^{-1}({\w},{\b}) + 2\frac{ \bar K({\b_1} ,\tilde{\w})}{ \bar K({\b_1}, {\w}) \bar K({\b} ,\tilde{\w})} \pi^{\mathbb Z^2}_\rho[\tilde{e}_0,e_1] \right)\\
			&=-\frac{ (\pi -\arg(\Omega_c-1))(\pi-2\arg(\Omega_c))}{\pi^2} - \frac{\pi -2\arg(\Omega_c-1)+2 \arg (\Omega_c)}{\pi^2}\Im \Omega_c .
		\end{split}
	\end{equation}

Denote by $Q=-\pi^2 \eqref{eq:bracket}$. Let $\Omega_c=r e^{\I \theta}$ and $\phi=\phi(r,\theta)=\arg(\Omega_c-1)$. First notice that $\lim_{r\to 0} Q=0$ since $\phi\to\pi$. To see \eqref{eq:positivity} it is thus enough to verify $\frac{dQ}{dr}\geq 0$. Using $\frac{d\phi}{dr}=-\frac{\sin\theta}{|\Omega_c-1|^2}$ we have
\begin{equation}
\frac{dQ}{dr} = \frac{\sin\theta}{|\Omega_c-1|^2} P,\textrm{ with } P=\pi-2\theta+2 r \sin\theta+(\pi+2\theta-2\phi)|\Omega_c-1|^2.
\end{equation}
If we see that $P\geq 0$, then also $\frac{dQ}{dr}\geq 0$. Now, $\lim_{r\to 0} P = 0$, thus it is enough to verify $\frac{dP}{dr}\geq 0$. Using $\frac{d |\Omega_c-1|^2}{dr}=2(r-\cos\theta)$, we have
\begin{equation}\label{eqA.14}
\frac{dP}{dr}=4\sin\theta+2(\pi+2\theta-2\phi)(r-\cos\theta).
\end{equation}
For $\theta\in [\pi/2,\pi)$, $\phi-\theta<\pi/2$ and thus both terms in \eqref{eqA.14} are positive. For $\theta\in [0,\pi/2)$, $\pi+2\theta-2\phi$ and $r-\cos\theta$ are both increasing in $r$ with $\lim_{r\to 0}(\pi+2\theta-2\phi)=2\theta-\pi<0$ and $\lim_{r\to 0}(r-\cos\theta)=-\cos\theta<0$. Also, at $r=\cos\theta$ one has $\pi+2\theta-2\phi=0$. This implies that both term in the rhs.~of~\eqref{eqA.14} are positive, ending the proof of \eqref{eq:positivity}.

Finally, to find $\Im \Omega_c$ notice that from Lemma~\ref{lem:Z2slopes} we have $\arg \Omega_c= \psi_1=(\rho_1+1/2)\pi$ and
\begin{equation}
 \psi_2=(\rho_2+1/2)\pi= \arg( \Omega_c-1)-\arg( \Omega_c+1)= \arg((\Omega_c-1)/(\Omega_c+1))
\end{equation}
from which $\tan\psi_2=2\frac{\Im \Omega_c}{|\Omega_c|^2-1}$.
By using $|\Omega_c|\sin \psi_1 =\Im \Omega_c$ we have
\begin{equation}
\tan \psi_2= 2 \frac{ \Im \Omega_c}{ (\frac{\Im \Omega_c}{\sin \psi_1})^2 -1} ,
\end{equation}
and we can solve for $\Im \Omega_c$ as required.
\end{proof}

\section{Hessian of the speed of growth for square lattice}\label{app:Hessian}
To verify that the model belongs to the anisotropic KPZ class of
growth models in $2+1$ dimensions, one needs to verify that the
determinant of the Hessian of $v^{\mathbb Z^2}$ is $\leq 0$. To do
this, consider the speed of growth as function of
$\psi_1,\psi_2\in[0,\pi]^2$ given in (\ref{vz2}). Denote by
${\rm Hess}(\psi_1,\psi_2)$ the Hessian of
$v^{\mathbb Z^2}(\psi_1,\psi_2)$.  An explicit (a bit lengthy)
computation gives that $\det({\rm Hess})$ equals
\begin{equation}
W(\psi_1,\theta) = \frac{-5-e^{4\theta}(2+e^{2\theta})^2+2e^{2\theta}[(2+e^{2\theta})\cos(4\psi_1)+4 (1+\sinh(2\theta))\cos(2\psi_1)]}{(2\cosh(\theta))^4}
\end{equation}
where we used the variables $\sinh(\theta):=\sin(\psi_1)/\tan(\psi_2)$. Now $\theta$ spans all $\R$. For any fixed $\theta\in\R$, we have that
$W(\psi_1,\theta)=W(\pi-\psi_1,\theta)$, thus we can restrict to $\psi_1\in[0,\pi/2]$.

A computation gives
\begin{equation}
\begin{aligned}
W(0,\theta)&=-\frac{(e^{4\theta}+2e^{2\theta}-3)^2}{(2\cosh(\theta))^4}\leq 0\textrm{ with }=0\textrm{ only for }\theta=0,\\
W(\pi/2,\theta)&=-4 e^{4\theta}< 0,\\
W(\psi_1,-\infty)&=0,\quad W(\psi_1,\infty)=-\infty.
\end{aligned}
\end{equation}
Thus at the boundary of the domain for $\psi_1$ and $\theta$ the
Hessian is $\leq 0$. Assume that there is a point inside the domain
where $W>0$. Then there is a maximum inside the domain where
$\frac{dW}{d\psi_1}=\frac{dW}{d\theta}=0$. The only possible solutions
of $\frac{dW}{d\psi_1}=0$ for $\psi_1\in (0,\pi/2)$ is
$\cos(\psi_1)=\zeta:=\cosh(\theta)/\sqrt{2+e^{2\theta}}$. With this
value of $\psi_1$, we have however
$\frac{dW}{d\theta}=-\frac{8e^{5\theta}\cosh(\theta)(3+2\cosh(2\theta))}{(2+e^{2\theta})^2}<0$. Thus
there is no maximum for
$(\psi_1,\theta)\in (0,\pi/2)\times (-\infty,\infty)$.


\begin{thebibliography}{10}

\bibitem{BCF16}
A.~Borodin, I.~Corwin, and P.L. Ferrari.
\newblock Anisotropic (2+1)d growth and gaussian limits of q-whittaker
  processes.
\newblock {\em preprint: arXiv:1612:00321}, 2016.

\bibitem{BCT15}
A.~Borodin, I.~Corwin, and F.L. Toninelli.
\newblock {Stochastic heat equation limit of a $(2+1)$D growth model}.
\newblock {\em Commun. Math. Phys., to appear (arXiv:1601.02767)}, 2016.

\bibitem{BF08b}
A.~Borodin and P.L. Ferrari.
\newblock {Anisotropic KPZ growth in $2+1$ dimensions: fluctuations and
  covariance structure}.
\newblock {\em J. Stat. Mech.}, page P02009, 2009.

\bibitem{BF08}
A.~Borodin and P.L. Ferrari.
\newblock {Anisotropic Growth of Random Surfaces in $2+1$ Dimensions}.
\newblock {\em Comm. Math. Phys.}, 325:603--684, 2014.

\bibitem{BS08}
A.~Borodin and S.~Shlosman.
\newblock Gibbs ensembles of nonintersecting paths.
\newblock {\em Com. Math. Phys. (online first)}, 2009.

\bibitem{CF15}
S.~Chhita and P.L. Ferrari.
\newblock {A combinatorial identity for the speed of growth in an anisotropic
  KPZ model}.
\newblock {\em preprint: arXiv:1508.01665}, 2015.

\bibitem{CJY15}
Sunil Chhita, Kurt Johansson, and Benjamin Young.
\newblock Asymptotic domino statistics in the {A}ztec diamond.
\newblock {\em Ann. Appl. Probab.}, 25(3):1232--1278, 2015.

\bibitem{CLP98}
H.~Cohn, M.~Larsen, and J.~Propp.
\newblock The shape of a typical boxed plane partition.
\newblock {\em New York J. Math.}, 4:137--165, 1998.

\bibitem{Dub15}
J.~Dub{\'e}dat.
\newblock Dimers and families of {C}auchy-{R}iemann operators {I}.
\newblock {\em J. Amer. Math. Soc.}, 28(4):1063--1167, 2015.

\bibitem{Dui11}
M.~Duits.
\newblock {The Gaussian free field in an interlacing particle system with two
  jump rates}.
\newblock {\em Comm. Pure Appl. Math.}, 66:600--643, 2013.

\bibitem{HHA92}
T.~Halpin-Healy and A.~Assdah.
\newblock On the kinetic roughening of vicinal surfaces.
\newblock {\em Phys. Rev. A}, 46:3527--3530, 1992.

\bibitem{Kas61}
P.~W. {Kasteleyn}.
\newblock {The statistics of dimers on a lattice : I. The number of dimer
  arrangements on a quadratic lattice}.
\newblock {\em Physica}, 27:1209--1225, 1961.

\bibitem{Kas63}
P.W. Kasteleyn.
\newblock Dimer statistics and phase transitions.
\newblock {\em J. Math. Phys.}, 4:287--293, 1963.

\bibitem{Ken97}
R.~Kenyon.
\newblock Local statistics of lattice dimers.
\newblock {\em Ann. Inst. H. Poincar\'e Probab. Statist.}, 33(5):591--618,
  1997.

\bibitem{Ken04}
R.~Kenyon.
\newblock Height fluctuations in the honeycomb dimer model.
\newblock {\em Comm. Math. Phys.}, 281:675--709, 2008.

\bibitem{KenLectures}
R.~Kenyon.
\newblock Lectures on dimers.
\newblock In {\em Statistical mechanics}, volume~16 of {\em IAS/Park City Math.
  Ser.}, pages 191--230. Amer. Math. Soc., Providence, RI, 2009.

\bibitem{KOS03}
R.~Kenyon, A.~Okounkov, and S.~Sheffield.
\newblock Dimers and amoebae.
\newblock {\em Ann. of Math.}, 163:1019--1056, 2006.

\bibitem{KKK98}
H.-J. Kim, I.~Kim, and J.M. Kim.
\newblock {Hybridized discrete model for the anisotropic Kardar-Parisi-Zhang
  equation}.
\newblock {\em Phys. Rev. E}, 58:1144--1147, 1998.

\bibitem{laslier2015quickly}
B.~Laslier and F.L. Toninelli.
\newblock How quickly can we sample a uniform domino tiling of the $2{L}\times
  2{L}$ square via {G}lauber dynamics?
\newblock {\em Probability Theory and Related Fields}, 161(3-4):509--559, 2015.

\bibitem{Pet14}
Leonid Petrov.
\newblock Asymptotics of random lozenge tilings via {G}elfand-{T}setlin
  schemes.
\newblock {\em Probab. Theory Related Fields}, 160(3-4):429--487, 2014.

\bibitem{PS97}
M.~Pr{\"a}hofer and H.~Spohn.
\newblock {An Exactly Solved Model of Three Dimensional Surface Growth in the
  Anisotropic KPZ Regime}.
\newblock {\em J. Stat. Phys.}, 88:999--1012, 1997.

\bibitem{TF61}
H.~N.~V. Temperley and M~E. Fisher.
\newblock Dimer problem in statistical mechanics---an exact result.
\newblock {\em Philos. Mag. (8)}, 6:1061--1063, 1961.

\bibitem{Ton15}
F.L.~Toninelli.
\newblock A $(2+1)$-dimensional growth process with explicit stationary
  measures.
\newblock {\em To appear in Ann. Probab.}, arXiv:1503.05339, 2015.

\bibitem{Wol91}
D.E. Wolf.
\newblock Kinetic roughening of vicinal surfaces.
\newblock {\em Phys. Rev. Lett.}, 67:1783--1786, 1991.

\end{thebibliography}

\end{document}